\documentclass[11pt,twoside,reqno,psamsfonts]{amsart}

\usepackage[OT1]{fontenc}
\usepackage{type1cm}
\usepackage{amssymb}
\usepackage{esint}
\usepackage[left=2.3cm,top=3cm,right=2.3cm]{geometry}
\usepackage{version}

\numberwithin{equation}{section}

\theoremstyle{plain}
\newtheorem{theorem}{Theorem}[section]

\newtheorem{proposition}[theorem]{Proposition}
\newtheorem{lemma}[theorem]{Lemma}

\theoremstyle{remark}
\newtheorem{remark}[theorem]{Remark}
\newtheorem{example}[theorem]{Example}
\newtheorem*{ack}{Acknowledgement}

\theoremstyle{definition}

\newcommand{\BB}{\mathcal{B}}
\newcommand{\HH}{\mathcal{H}}
\newcommand{\QQ}{\mathcal{Q}}

\newcommand{\R}{\mathbb{R}}
\newcommand{\Q}{\mathbb{Q}}
\newcommand{\N}{\mathbb{N}}

\newcommand{\iii}{\mathtt{i}}
\newcommand{\jjj}{\mathtt{j}}

\newcommand{\eps}{\varepsilon}
\newcommand{\roo}{\varrho}

\renewcommand{\atop}[2]{\genfrac{}{}{0pt}{}{#1}{#2}}

\DeclareMathOperator{\esssup}{ess\,sup}
\DeclareMathOperator{\essinf}{ess\,inf}

\DeclareMathOperator{\dimloc}{dim_{loc}}
\DeclareMathOperator{\udimloc}{\overline{dim}_{loc}}
\DeclareMathOperator{\ldimloc}{\underline{dim}_{loc}}

\DeclareMathOperator{\dimhomo}{dim_{hom}}

\DeclareMathOperator{\udim}{\overline{dim}}
\DeclareMathOperator{\ldim}{\underline{dim}}

\DeclareMathOperator{\dist}{dist}
\DeclareMathOperator{\diam}{diam}
\DeclareMathOperator{\por}{por}
\DeclareMathOperator{\spt}{spt}

\begin{document}

\title{Local homogeneity and dimensions of measures}

\author{Antti K\"aenm\"aki}
\author{Tapio Rajala}
\address{Department of Mathematics and Statistics \\
         P.O. Box 35 (MaD) \\
         FI-40014 University of Jyv\"askyl\"a \\
         Finland}
\email{antti.kaenmaki@jyu.fi}
\email{tapio.m.rajala@jyu.fi}

\author{Ville Suomala}
\address{Department of Mathematical sciences\\
P.O Box 3000\\
FI-90014 University of Oulu\\
Finland }
\email{ville.suomala@oulu.fi}

\thanks{The authors acknowledge the support of the Academy of Finland, projects \#114821, \#126976, and \#137528}
\date{\today}

\begin{abstract}
We introduce two new concepts, local homogeneity and local $L^q$-spectrum, both of which are tools that can be used in studying the local structure of measures. The main emphasis is given to the examination of local dimensions of measures in doubling metric spaces. As an application, we reach a new level of generality and obtain new estimates in the study of conical densities and porous measures.
\end{abstract}

\maketitle

\tableofcontents

\section{Introduction} \label{sec:intro}

In geometric measure theory, it is common to encounter problems of the
following type: given a measure $\mu$ and a set $A$ of positive/full
$\mu$-measure, we have some
local geometric information (on density, porosity, tangent measures, etc.) around all
points of the set (or in a set of positive/full measure) and we want
to gain some global information (on dimension, rectifiability,
measure, etc.) from this. For example, if the set is porous in the
sense that it contains large holes of fixed relative size around all of its
points in all small scales, it is reasonable to estimate
the dimension of the set from above using this information, see
\cite{Trocenko1981, Mattila1988, Salli1991, Mattila1995, KoskelaRohde1997, BeliaevSmirnov2002, KaenmakiSuomala2004, JarvenpaaJarvenpaaKaenmakiSuomala2005, Rajala2009}. Thus, if we knew how the set (or a
measure) is distributed in small balls, we would be able to bound its
dimension. On the other hand, if $\mu$ is a
measure of given dimension on a Euclidean space, it is a
classical problem to estimate how it is distributed in
different directions or cones,
see \cite{Besicovitch1929, Besicovitch1938, Marstrand1954,
  Federer1969,
  Salli1985, Falconer1985, Mattila1988, Mattila1995, Lorent2003, Suomala2005a,
  KaenmakiSuomala2004, KaenmakiSuomala2008,
  CsornyeiKaenmakiRajalaSuomala2010}.

In the study of fractals and dynamical systems, it is natural to analyse
properties of measures using globally observable parameters arising
from the asymptotic behaviour of the system, such as the Lyapunov exponents.
The entropy and $L^q$-dimensions are concepts that measure the
average distribution of the measure.
In many cases, these global
characteristics can then be related to the local
regularity properties of the measure such as exact dimensionality and
also to the values of the local dimension maps, see \cite{Young1982, Cutler1990,
Ngai1997, Heurteaux1998, BatakisHeurteaux2002, FanLauRao2002}.

In this
article, the most important objects of interest
are the upper and lower local dimensions of measures. Large part of the analysis on
measures aims at estimating these dimensions. The essential suprema
and infima of the local dimensions lead to the upper and lower
Hausdorff and packing dimensions of the measure whereas
investigating the level set structure of the local dimension maps
leads to multifractal analysis. The purpose of this article
is to introduce two new concepts, local homogeneity
and local $L^q$-spectrum.
Both of these concepts are tools that can be used in
studying the local structure of measures.

The local homogeneity and local $L^q$-spectrum are
of different nature since the order of taking limits in their definitions is different. In defining the local homogeneity, we
first let the scale tend to zero and only after that increase the
resolution. This allows us to handle non-uniform properties, like porosity, with
ease. On the other hand, the local $L^q$-spectrum sees some slight
differences in the behaviour of the measure to which the local
homogeneity is blind. This difference is made manifest in examples in \S \ref{sec:examples}.

We will next describe our main results. For notation and definitions of
the basic concepts, we refer to \S \ref{sec:notation} below.
In \S \ref{sec:local_Lq}, we introduce local versions of the
classical $L^q$-spectra and dimensions. Using these concepts, in \S \ref{sec:Lqresults}--\ref{sec:entropydim} we
obtain local metric space versions of the results of \cite{Ngai1997,
  Heurteaux1998, Olsen2000, FanLauRao2002} on the relations between the
Hausdorff, entropy, packing, and $L^q$-dimensions for measures in
Euclidean spaces.
In
Theorems \ref{thm:homodim} and \ref{thm:main}, we will prove our main results concerning
the local homogeneity of measures. We show that for any locally finite Borel regular
measure $\mu$, the upper local dimension $\udimloc(\mu,x)$  is bounded
from above by the local homogeneity dimension $\dimhomo(\mu,x)$ at
$\mu$-almost all points. Here $\dimhomo(\mu,x)$ is the
infimum of  exponents $s$ so that ``large parts'' of $B(x,r)$ in terms of
$\mu$ can be covered by $\delta^{-s}$ balls of radius $\delta r$ for all small
$r,\delta>0$; see \eqref{eq:homo_def} for a detailed definition. Using
our results on the local homogeneity, we will obtain new
estimates on the dimension of porous measures, see Theorems
\ref{thm:kpor} and \ref{thm:metricporo}. In particular, these results
settle problems left open in
\cite{JarvenpaaJarvenpaaKaenmakiRajalaRogovinSuomala2007,
  RajalaSmirnov2007}. As another application of the local homogeneity
estimates, we obtain in Theorem \ref{thm:conical} a new upper conical density result for measures
with large packing dimension. This improves a result of
\cite{CsornyeiKaenmakiRajalaSuomala2010} where a corresponding statement
was proved for the Hausdorff dimension.

Although the definitions of local homogeneity and local $L^q$-spectrum
make sense in any metric space in which balls are totally bounded, we
will consider only doubling metric spaces
since the doubling condition is needed in most of our proofs.

In many recent studies the relations between the dimension and geometry
of measures in Euclidean spaces are studied using a probabilistic approach
and the dyadic self-similar structure of $\R^n$. For instance, see
\cite{Heurteaux2007, JarvenpaaJarvenpaa2005, HochmanShmerkin2012, Shmerkin2012, SahlstenShmerkinSuomala2011}.
Since we work in a general doubling metric space, our approach is different
and slightly less probabilistic. We remark that the paper \cite{SahlstenShmerkinSuomala2011}
was mostly inspired by the present work.

\section{Preliminaries} \label{sec:notation}

\subsection{Basic notation}
In writing down constants we often use notation such as $c=c(\cdots)$
to emphasize that the constant depends only on the parameters listed
inside the parentheses.

We work on a metric space $(X,d)$ which we usually assume to be \emph{doubling}.
This means that there is $N=N(X)\in\N$ (\emph{the doubling constant of}) $X$ such that any closed ball $B(x,r) = \{ y
\in X : d(x,y) \le r \}$ with centre $x \in X$ and radius $r>0$ can be
covered by $N$ balls of radius $r/2$.
Since we use only one distance $d$ in the space $X$, we simply denote $(X,d)$ by $X$.

Notice that even if $x \ne y$ or $r\neq t$, it may happen that
$B(x,r)=B(y,t)$. For notational convenience, we keep to the convention that each ball comes with a fixed center and radius.
This makes it possible to use notation such as
$5B=B(x,5r)$ without referring to the centre or radius of the ball $B=B(x,r)$.

In this article, \emph{a measure} exclusively refers to a nontrivial Borel regular (outer) measure
defined on all subsets of $X$ so that bounded sets have finite
measure.

We call any countable collection $\BB$ of pairwise disjoint closed
balls a \emph{packing}. It is called a \emph{packing of $A$} for a
subset $A \subset X$ if the centres of the balls of $\BB$ lie in the
set $A$, and it is a \emph{$\delta$-packing} for $\delta > 0$ if all
of the balls in $\BB$ have radius $\delta$. A $\delta$-packing $\BB$
of $A$ is termed \emph{maximal} if for every $x \in A$ there is $B \in
\BB$ so that $B(x,\delta) \cap B \ne \emptyset$. Note that if $\BB$ is
a maximal $\delta$-packing of $A$, then $2\BB$ covers $A$.
Here $2\BB = \{ 2B : B \in \BB \}$.

Observe that a doubling metric space is separable. Hence for each
$\delta>0$ and $A \subset X$ there exists a maximal $\delta$-packing
of $A$. Moreover, the $5r$-covering theorem is applicable in every doubling metric space; see \cite[Theorem 2.1]{Mattila1995}.

Instead of $\delta$-packings defined above, the theory developed in this paper
could be presented by using $\delta$-separated sets, i.e.\ sets $\{x_i\}\subset A$
for which $d(x_i,x_j)>\delta$ whenever $x_i\neq x_j$. Yet another option would be to
define the necessary concepts using partitions or generalised dyadic cubes. We chose
the packing approach mainly because of personal taste and since we wanted our packing
balls to be geometrically (and not only algebraically) disjoint. The partition definition
is sometimes more useful in computations. In \cite{KaenmakiRajalaSuomala2012b}, we
use that approach to develop some multifractal analysis in metric spaces.

The doubling property can be stated in several equivalent ways. For instance, the following formulations are sometimes convenient.
The proof is a simple exercise (see e.g.\ \cite{Luukkainen1998, Heinonen2001}).

\begin{lemma} \label{thm:covering_thm}
  For a metric space $X$, the following statements are equivalent:
\begin{enumerate}
  \item $X$ is doubling. \label{covering1}
  \item There are $s>0$ and $c>0$ such that for all $R>r>0$ any ball of radius $R$ can be
        covered by $c(r/R)^{-s}$ balls of radius $r$. \label{covering2}
  \item There are $s>0$ and $c>0$ such that if $R>r>0$ and $\BB$ is an $r$-packing of a
        closed ball of radius $R$, then the cardinality of $\BB$ is at most $c(r/R)^{-s}$. \label{covering3}
  \item For every $0<\lambda<1$ there is a constant $M=M(X,\lambda) \in
        \N$, satisfying the following: If $\BB$ is a collection of closed
        balls of radius $\delta > 0$ so that $\lambda\BB$ is pairwise disjoint,
        then there are $\delta$-packings $\{ \BB_1,\ldots,\BB_M \}$ so that
        $\BB=\bigcup_{i=1}^M\BB_i$. \label{covering4}
  \item There is $M=M(X)\in\N$ such that if $A\subset X$ and $\delta>0$, then there are
        $\delta$-packings of $A$, $\BB_1,\ldots,\BB_M$ whose union covers $A$. \label{covering5}
\end{enumerate}
\end{lemma}

\begin{remark}\label{rem:covering}
(1) It follows by elementary arguments that $s=\log_2 N$ will do in \eqref{covering2} and \eqref{covering3}.
The infimum over all admissible exponents $s$ in \eqref{covering2} and \eqref{covering3} is usually called
the \emph{Assouad dimension} of $X$ (see \cite{Luukkainen1998, Heinonen2001}). Thus, doubling metric spaces
are precisely the metric spaces with finite Assouad dimension.

(2) Observe that \eqref{covering5} is Besicovitch's covering theorem (\cite[\S 2.7]{Mattila1995})
for balls with equal radius. The following consequence of \eqref{covering5} is sometimes very
useful: If $\delta>0$, $\mu$ is a measure on $X$ and $A\subset X$, then there is a $\delta$-packing $\BB$ of $A$ such that
\begin{equation}\label{packingestimate}
\sum_{B\in\BB}\mu(B)\ge c\mu(A).
\end{equation}
Here $c>0$ depends only on the doubling constant $N$.
\end{remark}

We say that a measure $\mu$ on $X$ is \emph{doubling} if there is a constant $c \ge 1$ so that
\begin{equation*}
  0<\mu\bigl( B(x,2r) \bigr) \le c\mu\bigl( B(x,r) \bigr) < \infty
\end{equation*}
for all $x \in X$ and $r>0$. A complete doubling metric space always supports doubling measures;
see \cite{VolbergKonyagin1984, VolbergKonyagin1987, LuukkainenSaksman1998, Wu1998, KaenmakiRajalaSuomala2012a}.
Recall that the \emph{support} of a measure $\mu$, denoted by $\spt(\mu)$, is the smallest closed subset
of $X$ with full $\mu$-measure. Furthermore, we say that
a measure $\mu$ on $X$ is \emph{$s$-regular} (for $s>0$) if there are
constants $a,b>0$ so that
\begin{equation*}
  a r^s \le \mu\bigl( B(x,r) \bigr) \le b r^s
\end{equation*}
for all $x \in \spt(\mu)$ and $0<r\le \diam(X)$. It is clear that each $s$-regular
measure is doubling. A metric space $X$ is called
\emph{$s$-regular} if it carries an $s$-regular measure $\mu$ with
$\spt(\mu)=X$.
In this case, a simple volume argument can be used to verify the conditions
\eqref{covering2} and \eqref{covering3} of Lemma \ref{thm:covering_thm}. Therefore
an $s$-regular metric space is doubling. More generally, each metric space carrying
a doubling measure is a doubling metric space.

A measure $\mu$ on $X$ has the \emph{density point property} if
\begin{equation}
  \lim_{r \downarrow 0} \frac{\mu\bigl( A \cap B(x,r) \bigr)}{\mu\bigl( B(x,r) \bigr)} = 1
\end{equation}
for $\mu$-almost all $x \in A$ whenever $A \subset X$ is
$\mu$-measurable.
In general, the density point property is not necessarily valid for all measures in a
doubling metric space; see Example \ref{ex:nodpp0}. Nevertheless, in the proofs, it can
be often replaced by the following weaker result.

\begin{lemma}\label{lemma:weakdpp}
If $\mu$ is a measure on a separable metric space $X$ and $A\subset X$ is $\mu$-measurable, then
\[\lim_{r\downarrow0}\frac{\mu\bigl( B(x,r) \setminus A \bigr)}{\mu\bigl( B(x,5r) \bigr)}=0\]
for $\mu$-almost all $x\in A$.
\end{lemma}

\begin{proof}
Define $E_\varepsilon=\{x\in A : \limsup_{r\downarrow0} \mu\bigl( B(x,r) \setminus A \bigr)/\mu\bigl( B(x,5r) \bigr) > \varepsilon\}$ for all $\eps>0$. The claim follows if we can show that $\mu(E_\varepsilon)=0$ for all $\varepsilon>0$. Fix $\varepsilon>0$ and for $\eta>0$, let $G_\eta$ be an open set containing $E_\varepsilon$ such that $\mu(G_\eta\setminus E_\varepsilon)<\eta$. Applying the $5r$-covering theorem for the collection $\{B(x,r) : x\in E_\varepsilon \text{ and $r>0$ such that } B(x,r)\subset G_\eta \text{ and } \mu\bigl( B(x,r)\setminus A \bigr) > \varepsilon\mu\bigl( B(x,5r) \bigr)\}$,
we obtain a disjoint subcollection $\BB$ such that $5\BB$ covers $E_\varepsilon$. Thus
\begin{align*}
\varepsilon \mu(E_\varepsilon)\le\varepsilon\sum_{B\in\BB}\mu(5B)<\sum_{B\in\BB}\mu(B\setminus A)\le\mu(G_\eta\setminus A)\le\mu(G_\eta\setminus E_\eps)<\eta.
\end{align*}
Letting $\eta\downarrow 0$ implies $\mu(E_\varepsilon)=0$, as required.
\end{proof}

\begin{remark} \label{rem:doubling}
(1) The constant $5$ in Lemma \ref{lemma:weakdpp} can be replaced by any constant $C>2$.
This is because in the $5r$-covering theorem, we may replace $5$ by any such $C$.
Furthermore, if Besicovitch's covering theorem holds in $X$, then the constant $5$ in
Lemma \ref{lemma:weakdpp} can be replaced by $1$. This can be seen just by applying
Besicovitch's covering theorem (instead of the $5r$-covering theorem) in the proof of
Lemma \ref{lemma:weakdpp}. In particular, this observation shows that in Euclidean spaces,
every measure has the density point property.

(2) The following upper density point property is true for all measures in any doubling metric
space $X$: If $\mu$ is a measure on $X$ and $A \subset X$ is $\mu$-measurable, then
\begin{equation*}
  \limsup_{r \downarrow 0} \frac{\mu\bigl( A \cap B(x,r) \bigr)}{\mu\bigl( B(x,r) \bigr)} = 1
\end{equation*}
for $\mu$-almost all $x \in A$. This follows from Lemma \ref{lemma:weakdpp} and the fact that even if
a measure is not doubling, it has arbitrary small doubling scales at each typical point,
see e.g. \cite[Lemma 2.2]{CsornyeiKaenmakiRajalaSuomala2010}.
\end{remark}

\subsection{Local dimensions}
We are mostly interested in estimating the \emph{upper and lower local dimensions of the measure $\mu$ at $x$} defined by
\begin{align*}
  \udimloc(\mu,x) &= \limsup_{r \downarrow 0} \log\mu\bigl( B(x,r)  \bigr)/\log r, \\
  \ldimloc(\mu,x) &= \liminf_{r \downarrow 0} \log\mu\bigl( B(x,r)  \bigr)/\log r,
\end{align*}
respectively. If the upper and lower dimensions agree, we call their
mutual value the \emph{local dimension of the measure $\mu$ at $x$}
and write $\dimloc(\mu,x)$ for this common value.

\begin{remark}
  (1) If $\mu$ is an $s$-regular measure, then trivially $\dimloc(\mu,x) = s$ for all $x \in \spt(\mu)$.

  (2) If $A$ is a Borel set, then $\udimloc(\mu|_A,x) = \udimloc(\mu,x)$ and
      $\ldimloc(\mu|_A,x) = \ldimloc(\mu,x)$ for $\mu$-almost all $x \in A$.
      This can be proven similarly as Lemma \ref{lemma:weakdpp} once we observe that if the statement
      fails, then there is $\varepsilon>0$ such that
      $\limsup_{r\downarrow 0}r^{\varepsilon}\mu\bigl(B(x,r)\bigr)/\mu\bigl(A \cap B(x,5r)\bigr)>0$ in a set of positive measure.
\end{remark}

\subsection{Local $L^q$-spectrum and $L^q$-dimensions}\label{sec:local_Lq}

Let $\mu$ be a measure on $X$, $A\subset X$ a bounded set, $q\in\R$, and $r>0$.
The \emph{(global) $L^q$-spectrum of $\mu$ on $A$} is defined by
\begin{equation} \label{eq:taudeffi}
  \tau_q(\mu,A)=\liminf_{\delta \downarrow 0}
  \frac{\log S_{q}(\mu,A,\delta)}{\log\delta},
\end{equation}
where
\begin{equation*}
  S_{q}(\mu,A,\delta) = \sup\biggl\{ \sum_{B \in \BB} \mu(B)^q :
  \BB \text{ is a $\delta$-packing of } A\cap \spt(\mu) \biggr\}
\end{equation*}
is the \emph{$L^q$-moment sum of $\mu$ on $A$ at the scale $\delta$}.
Furthermore, the \emph{local $L^q$-spectrum of $\mu$ at $x$} is
\begin{equation} \label{eq:spectrum}
  \tau_q(\mu,x) = \lim_{r \downarrow 0} \tau_q\bigl( \mu, B(x,r) \bigr).
\end{equation}

Given $A\subset X$ and $q \ne 1$, we define the \emph{(global) $L^q$-dimension of $\mu$ on $A$} by setting
\begin{equation*}
  \dim_q(\mu,A) = \tau_q(\mu,A)/(q-1)
\end{equation*}
and the \emph{local $L^q$-dimension of $\mu$ at $x$} by
\begin{equation*}
  \dim_q(\mu,x) = \lim_{r \downarrow 0} \dim_q\bigl( \mu,B(x,r) \bigr) = \tau_q(\mu,x)/(q-1).
\end{equation*}
We also denote $\tau_q(\mu) = \tau_q(\mu,X)$ and $\dim_q(\mu) = \dim_q(\mu,X)$ in the case $\spt(\mu)$ is bounded.

\begin{remark} \label{rem:tau}
  (1) The limit in \eqref{eq:spectrum} exists as $S_{q}(\mu,A,\delta) \le S_{q}(\mu,B,\delta)$
 whenever $\delta>0$ and $A\subset B$. The use of $\liminf$ in \eqref{eq:taudeffi} guarantees
 the concavity of the $L^q$-spectrum; see Lemma \ref{tauprop}\eqref{tauconcave}.

   (2) If $q\ge0$ and $A$ is closed, then the definition of $\tau_q(\mu,\cdot)$ does not
 change if we ignore $\spt(\mu)$ in the definition of $S_{q}(\mu,\cdot)$. That is, we can
 repeat the definition with
\begin{equation*}
  S_{q}(\mu,A,\delta) = \sup\biggl\{ \sum_{B \in \BB} \mu(B)^q :
  \BB \text{ is a $\delta$-packing of } A \biggr\}
\end{equation*}
(if $q=0$, we interpret $0^q=0$). Also, if $(\delta_n)_{n=1}^\infty$ is a decreasing sequence
tending to $0$ with $\lim_{n\rightarrow\infty}\log\delta_{n+1}/\log\delta_n=1$, then it follows
from Lemma \ref{thm:covering_thm}\eqref{covering5} that the $\liminf$ in the definition of
$\tau_q$ may be taken along the sequence $(\delta_n)_{n=1}^\infty$.
These simple facts will be used frequently.

  (3) If $\mu$ is an $s$-regular measure on $X$ with $\spt(\mu)=X$, then $\dim_q(\mu,A) = s$ for
  all bounded $A\subset X$ with $\mu(A)>0$ and, consequently, $\dim_q(\mu,x)=\dimloc(\mu,x)$
  for all $x \in X$. Indeed, given $q\in\R$, we find constants $0<c_1(A)<c_2(A)<\infty$ so that
  $c_1 \delta^{s(q-1)} \le S_{q}(\mu,A,\delta) \le c_2 \delta^{s(q-1)}$ for all $0<\delta<1$.
  This implies $\tau_q(\mu,A)=s(q-1)$ and thus $\dim_q(\mu,A)=s$.

  (4) There are measures for which $\dim_q(\mu,x)$ is constant almost everywhere, but this constant
  is not the same as $\dim_q(\mu)$; see Examples \ref{ex:eka} and \ref{ex:toka}.

  (5) Recall that for any Borel set $A$ the restriction measure $\mu|_A$ has the same upper and
 lower local dimension as the original measure $\mu$ for $\mu$-almost all points in $A$.
 This is not true for the $L^q$-dimension.
  As an example in the case $q < 1$, take $\mu = \mathcal{L}^2 + \HH^1|_L$
  on $[0,1]^2$, where $\mathcal{L}^2$ is the Lebesgue measure and
  $\HH^1|_L$ is the length measure on a line $L \subset [0,1]^2$.
  Now there exist constants $c_1,c_2 > 0$ so that for every $r>0$ we have
 $S_q\bigl( \mu,B(x,r),\delta \bigr) = c_1r^2\delta^{2(q-1)}$ and
 $S_q\bigl( \mu|_L,B(x,r),\delta \bigr) = c_2r\delta^{q-1}$ for all $\delta > 0$ small enough.
 Thus $\tau_q(\mu,x) = \tau_q(\mu)$ and $\tau_q(\mu|_L,x) = \tau_q(\mu|_L)$ for all $x \in L$.
 Since $\spt(\mu|_L) = \spt(\mu) \cap L$, we also have $\tau_q(\mu|_L) = \tau_q(\mu,L)$. Therefore,
  \begin{equation*}
    \tau_q(\mu,x) = \tau_q(\mu) = 2(q-1) < q-1 = \tau_q(\mu|_L) = \tau_q(\mu,L)
  \end{equation*}
  and
  \begin{equation*}
    \dim_q(\mu,x) = 2 > 1 = \dim_q(\mu|_L,x)
  \end{equation*}
  for all $x \in L$. For $q > 1$ we define a measure on the real line by letting
 $\nu = \mathcal{L}^2 + \sum_{n\in\N}2^{-n} \delta_{q_n}$, where $\{q_1,q_2,\ldots\}$ is
 an enumeration of the rationals. Then $\dim_q(\nu,x)=0$ while $\dim_q(\nu|_{\R\setminus \Q},x)=1$ for all $x\in\R$.
\end{remark}

We list some of the basic properties of the $L^q$-spectrum in the following lemmas.

\begin{lemma}\label{tauprop}
If $\mu$ is a measure on a doubling metric space $X$, $A\subset X$ is a bounded set with $\mu(A)>0$,
$q_0 = \inf\{ q \in \R : \tau_q(\mu,A) > -\infty \}$, and $s$ as in Lemma \ref{thm:covering_thm}(2)--(3), then
  \begin{enumerate}
    \item $\tau_1(\mu,A)=0$, \label{tau10}

    \item $\min\{0,(q-1)s\} \leq \tau_q(\mu,A) \leq \max\{0,(q-1)s\}$ for all $0 \le q < \infty$, \label{taubounded}

    \item $0\le \dim_q(\mu,A)\le s$ for all $0 \le q < \infty$ with $q \ne 1$, \label{dimqbounded}

    \item the mapping $q\mapsto\tau_q(\mu,A)$ is concave on $(q_0,\infty)$. \label{tauconcave}

    \item the mapping $q\mapsto\dim_q(\mu,A)$ is continuous and decreasing on both $(q_0,1)$ and $(1,\infty)$. \label{taucont}
\end{enumerate}
Furthermore, if $x\in \spt(\mu)$, then all the claims above remain true if $\tau_q(\mu,A)$ is replaced by
$\tau_q(\mu,x)$ and $\dim_q(\mu,A)$ by $\dim_q(\mu,x)$.
\end{lemma}

\begin{proof}
We prove the claims for $\tau_q(\mu,A)$. The statements for $\tau_q(\mu,x)$ follow by simply taking $A=B(x,r)$
and letting $r\downarrow0$. It suffices to show (2) and (4) since the other claims follow easily from these.
Fix $a\in A$ and define $U=B\bigl( a,\diam(A)+1 \bigr)$.

If $0<\delta<1$ and $\BB$ is any $\delta$-packing of $A$, then Lemma \ref{thm:covering_thm}\eqref{covering3}
gives $M \le C\delta^{-s}$, where $M$ is the cardinality of $\BB$. Therefore H\"older's inequality implies
\begin{equation*}
  \sum_{B \in\BB} \mu(B)^q \le
  \begin{cases}
    \mu(U)^q M^{1-q} \le C^{1-q}\mu(U)^q \delta^{s(q-1)}, &\text{if } 0 \le q \le 1, \\
    \mu( U)^q, &\text{if } q \ge 1.
  \end{cases}
\end{equation*}
In addition, if $\BB$ satisfies \eqref{packingestimate}, then we estimate
\begin{equation*}
  \sum_{B \in\BB} \mu(B)^q \ge
  \begin{cases}
    c^q\mu\bigl( A \bigr)^q, &\text{if } q \le 1, \\
    c^q\mu( A )^q M^{1-q} \ge c^qC^{1-q}\mu(A)^q \delta^{s(q-1)}, &\text{if } q \ge 1.
  \end{cases}
\end{equation*}
The claim (2) follows by taking logarithms and letting $\delta\downarrow 0$.

To show (4), let $\BB$ be a $\delta$-packing of $A \cap \spt(\mu)$. For every $q,p\ge q_0$ and $\lambda\in(0,1)$ we have
  \begin{equation}\label{convex}
    \sum_{B \in \BB} \mu(B)^{\lambda q + (1-\lambda)p} \le
    \biggl(\sum_{B \in \BB} \mu(B)^{q} \biggr)^\lambda
    \biggl(\sum_{B \in \BB} \mu(B)^{p} \biggr)^{1-\lambda}
  \end{equation}
by H\"older's inequality. The proof follows.
\end{proof}

\begin{lemma}\label{prop:tauloc}
If $\mu$ is a measure on a compact doubling metric space $X$, then
\begin{equation*}
\tau_q(\mu) = \min\{ \tau_q(\mu, x) : x \in \spt(\mu) \}
\end{equation*}
for every $q\in\R$. In particular,
\begin{equation*}
  \dim_q(\mu) = \begin{cases}
                  \max\{ \dim_q(\mu,x) : x \in \spt(\mu) \}, &\text{if } q < 1, \\
                  \min\{ \dim_q(\mu,x) : x \in \spt(\mu) \}, &\text{if } q > 1.
               \end{cases}
\end{equation*}
\end{lemma}

\begin{proof}
According to Remark \ref{rem:tau}(1), we have $\tau_q(\mu) \le \tau_q(\mu,x)$
for every $x \in\spt(\mu)$. Since the second claim follows immediately from the
first one, it remains to show that there exists $x \in \spt(\mu)$ for which $\tau_q(\mu,x) \le \tau_q(\mu)$.

First we cover $\spt(\mu)$ with finitely many balls
$\{B(y_i,\tfrac12)\}_{i=1}^{k_1}$, $y_i\in \spt(\mu)$. Then, for
every $j$ and $\delta >0$, we have
\begin{equation}\label{eq:taucomparison}
\begin{split}
  S_q\bigl( \mu,B(y_j,\tfrac12),\delta \bigr) &\le S_q(\mu,X,\delta) \le \sum_{i=1}^{k_1}S_q\bigl( \mu,B(y_i,\tfrac12),\delta \bigr) \\ &\le k_1 \max_{i \in \{ 1,\ldots,k_1 \}}S_q\bigl( \mu,B(y_i,\tfrac12),\delta \bigr).
\end{split}
\end{equation}
Let $(\delta_j)_{j=1}^\infty$ be a decreasing sequence tending to zero so that
\begin{equation*}
  \lim_{j \to \infty}\frac{\log S_q(\mu,X,\delta_j)}{\log \delta_j} = \liminf_{\delta \downarrow 0}\frac{\log S_q(\mu,X,\delta)}{\log \delta}=\tau_q(\mu).
\end{equation*}
Then for every $j \in \N$,
choose $i_j\in\{1,\ldots,k_1\}$ so that
\[
 S_q\bigl( \mu,B(y_{i_j},\tfrac12),\delta_j \bigr) = \max_{i \in \{ 1,\ldots,k_1 \} }S_q\bigl( \mu,B(y_i,\tfrac12),\delta_j \bigr).
\]
Now for some $i \in \{1,
\dots, k_1\}$ the set $\{j \in \N : i_j = i\}$ is
infinite. Considering a suitable subsequence of
$(\delta_j)_{j=1}^\infty$ and using \eqref{eq:taucomparison}, we get
\begin{equation*}
 \liminf_{\delta \downarrow 0}\frac{\log
   S_q\bigl( \mu,B(x_1,\tfrac12),\delta \bigr)}{\log \delta} = \tau_q(\mu),
\end{equation*}
where $x_1 = y_i$.

Next we repeat the above argument by replacing
$\tfrac12$ with $\tfrac14$ and $\spt(\mu)$ by $\spt(\mu)\cap
B(x_1,\tfrac12)$.
Then we find $x_2 \in B(x_1,\tfrac12)$ so that
\begin{equation*}
 \liminf_{\delta \downarrow 0}\frac{\log
   S_q\bigl( \mu,B(x_2,\tfrac14),\delta \bigr)}{\log \delta} = \liminf_{\delta
   \downarrow 0}\frac{\log S_q\bigl( \mu,B(x_1,\tfrac12),\delta \bigr)}{\log
   \delta} = \tau_q(\mu).
\end{equation*}
Continuing inductively, we find a sequence $x_i\in \spt(\mu)$ with
$d(x_{i+1},x_i)\leq 2^{-i}$ and
\begin{equation*}
 \liminf_{\delta \downarrow 0}\frac{\log
   S_q\bigl( \mu,B(x_i,2^{-i}),\delta \bigr)}{\log \delta} = \tau_q(\mu)
\end{equation*}
for every $i\in\N$. Since $\spt(\mu)$ is compact, for $x = \lim_{i\to
  \infty} x_i$, we
eventually get
\begin{equation*}
\liminf_{\delta\downarrow 0}\frac{\log S_q\bigl( \mu,B(x,2^{-i+2}),\delta \bigr)}{\log
  \delta}\leq \liminf_{\delta\downarrow 0}\frac{\log
  S_q\bigl( \mu,B(x_i,2^{-i}),\delta \bigr)}{\log\delta}
\end{equation*}
for all $i \in \N$ and thus $\tau_q(\mu,x)\leq \tau_q(\mu)$.
\end{proof}

\begin{remark}
  If $\mu$ is a measure on a doubling metric space $X$ and $A \subset X$ is compact,
then an easy modification of the above proof shows that for each $q \in \N$ there
exists $x \in A \cap \spt(\mu)$ so that $\tau_q(\mu,x) \le \tau_q(\mu,A)$. Recall also Remark \ref{rem:tau}(4). Then
    $\dim_q(\mu,A) \le \max\{ \dim_q(\mu,x) : x \in A \cap \spt(\mu) \}$
  for $q<1$ and
    $\dim_q(\mu,A) \ge \min\{ \dim_q(\mu,x) : x \in A \cap \spt(\mu) \}$
  for $q>1$.
\end{remark}

\subsection{Local homogeneity and homogeneity dimension}

Let $\mu$ be a measure on $X$, $x \in X$, and $\delta, \eps, r > 0$. Define for all $\Lambda > 1$
\begin{equation} \label{eq:homo_def3}
\begin{split}
   \hom_{\delta,\eps,r}^\Lambda(\mu,x) = \sup\{ \#\BB :\; &\BB \text{ is a $(\delta
r)$-packing of } B(x,r) \\ &\text{so that } \mu(B) > \eps\mu\bigl( B(x,\Lambda r)
\bigr) \text{ for all } B \in \BB \}
\end{split}
\end{equation}
and from this let the \emph{local $\delta$-homogeneity (with a parameter $\Lambda$) of a measure $\mu$ at $x$} be
\begin{equation} \label{eq:homo_def2}
   \hom_\delta^\Lambda(\mu,x) = \lim_{\eps \downarrow 0} \limsup_{r \downarrow 0}
\hom_{\delta,\eps,r}^\Lambda(\mu,x).
\end{equation}
The \emph{local homogeneity dimension (with a parameter $\Lambda$) of a measure $\mu$ at $x$} is then defined as
\begin{equation} \label{eq:homo_def}
   \dim_{\hom}^\Lambda(\mu,x) = \liminf_{\delta \downarrow 0}
\frac{\log\hom_\delta^\Lambda(\mu,x)}{-\log\delta},
\end{equation}
where we interpret $\log{0}=0$ to ensure $\dim_{\hom}^\Lambda(\mu,x)\ge0$.

\begin{remark}\label{homorems}
(1) The limit in \eqref{eq:homo_def2} exists as
$\hom_{\delta,\eps_2,r}^\Lambda(\mu,x) \le \hom_{\delta,\eps_1,r}^\Gamma(\mu,x)$ for all
$0<\eps_1<\eps_2$ and $\Lambda \ge \Gamma > 1$.

(2) The definition of $\dim_{\hom}^\Lambda$ is quite technical. It may be helpful
to compare it to the definition of the Assouad dimension given in Remark \ref{rem:covering}(1).
The local homogeneity dimension may be considered as a kind of local Assouad dimension for the measure $\mu$ around $x$:
it is the least possible exponent $s$ so that for all small $\delta,r>0$ the ball $B(x,r)$ has a
$\delta$-packing of size $\delta^{-s}$ such that the $\mu$ measure of the packing balls is comparable
to $\mu\bigl( B(x,\Lambda r) \bigr)$.

  (3) If $\mu$ is an $s$-regular measure on $X$ with $\spt(\mu)=X$, then
  \begin{equation*}
    \dim_{\hom}^\Lambda(\mu,x) = \dimloc(\mu,x) = s
  \end{equation*}
  for all $x \in X$.
  Indeed, a simple volume argument implies that for all $x\in X$, $r>0$ and $0<\delta<1$,
 we have $c_1\delta^{-s}\le \sup\{\#\BB : \BB \text{ is a $(\delta r)$-packing of } B(x,r) \} \le c_2 \delta^{-s}$.
On the other hand, if $\varepsilon=\varepsilon(\delta)>0$ is small, then we have
$\mu\bigl( B(y,\delta r) \bigr) > \varepsilon \mu\bigl( B(x,r) \bigr)$.

  (4) Let $\mu$ be a measure on $X$.
Then, for every $\mu$-measurable $A \subset X$, we have
  \begin{equation*}
    \dim_{\hom}^\Lambda(\mu|_A,x) = \dim_{\hom}^\Lambda(\mu,x)
  \end{equation*}
  for $\mu$-almost all $x \in A$. This is easily seen by combining Lemma \ref{lemma:weakdpp}
and Lemma \ref{prop:5} below with the estimates
$\mu\bigl( A \cap B(x, 5\Lambda r) \bigr) \ge \varepsilon\mu\bigl( B(x,\Lambda r) \bigr)$
and $\mu(B_i\cap A)\ge\mu(B_i)-\varepsilon\mu\bigl( B(x,5r) \bigr)$ for $B_i\subset B(x,r)$
and $r,\varepsilon>0$ small enough.
\end{remark}

The next lemma shows that at a typical point, the choice of the parameter $\Lambda$ in the
definition of homogeneity does not play any role. Therefore, in the applications, we may
choose a convenient value for $\Lambda$.

\begin{lemma}\label{prop:5}
If $\mu$ is a measure on a doubling metric space $X$ and $\Lambda>\Gamma>1$, then
$\dim_{\hom}^{\Lambda}(\mu,x) = \dim_{\hom}^{\Gamma}(\mu,x)$ for $\mu$-almost every $x \in X$.
\end{lemma}

\begin{proof}
 According to Remark \ref{homorems}(1), we have $\dim_{\hom}^\Lambda(\mu,x) \le \dim_{\hom}^\Gamma(\mu,x)$
 for all $x \in X$. The main point in the proof of
 the opposite inequality is the observation that if $\mathcal{B}$ is a $\delta$-packing of $B(x,r)$,
 then for constants $c_1=c_1(\Lambda,\Gamma)>0$ and $c_2=c_2(N,\Lambda,\Gamma)>0$ there are $y\in B(x,r)$
 and a $\delta$-packing $\mathcal{B}'\subset\mathcal{B}$ of $B(y,cr)$ such that $B(y,\Lambda c_1 r)\subset B(x,\Gamma r)$
 and $\#\mathcal{B}'\ge c_2\#\mathcal{B}$.

  In order to deliver full details of the proof, we assume to the contarary that there exist a set $A \subset X$ with $\mu(A)>0$ and $t>0$ so that
  \begin{equation*}
    \dim_{\hom}^\Lambda(\mu,x) < t < \dim_{\hom}^\Gamma(\mu,x)
  \end{equation*}
  for all $x \in A$. Let $c=(\Gamma-1)/2\Lambda\Gamma^q$ where $q \in \N$ is chosen so that
 $\Gamma^{q-1} \ge 5/(\Gamma-1)$. According to Lemma \ref{thm:covering_thm}(2) there exists
 $M \in \N$ such that a ball of radius $r$ can be covered by $M$ balls of radius $\min\{ c,\Gamma^{-q} \}r$ for all $r>0$.

  Going into a subset of $A$,
  if necessary, we find $r_0,\eps,\delta > 0$ so that $\delta < \Gamma^{-q}$,
  \begin{equation*}
    \hom^\Lambda_{\delta,\eps,r}(\mu,x) < \delta^{-t}/M^2
  \end{equation*}
  for every $0<r<r_0$ and $x \in A$, and
  \begin{equation*}
    \limsup_{r \downarrow 0} \hom^\Gamma_{c\delta,\eps,\Gamma^qr}(\mu,x) > \delta^{-t}
  \end{equation*}
  for all $x \in A$. Recalling Lemma \ref{lemma:weakdpp}, we may also assume that
  \begin{equation}\label{aa}
    \mu\bigl( B(x,\Gamma r) \setminus A \bigr) < \eps\delta^{-t}\mu\bigl( B(x,5\Gamma r) \bigr)/M^2
  \end{equation}
  for all $0<r<r_0$ and $x \in A$.

  Next we fix $x \in A$ and choose $0<r<r_0/\Gamma^qc$ so that
  \begin{equation*}
    \hom^\Gamma_{c\delta,\eps,\Gamma^qr}(\mu,x) > \delta^{-t}.
  \end{equation*}
  Since $A \cap B(x,\Gamma^qr)$ can be covered by $M$ balls of radius $r$ with centers in
$A \cap B(x,\Gamma^qr)$, we find $w \in A \cap B(x,\Gamma^q r)$ and a $(\Gamma^qc\delta r)$-packing
$\BB'$ of $B(w,r)$ so that $\#\BB' \ge \delta^{-t}/M$ and
\begin{equation}\label{bb}
\mu(B) > \eps\mu\bigl( B(x,\Gamma^{q+1} r) \bigr) \ge \eps\mu\bigl( B(w,5\Gamma r) \bigr)
\end{equation}
for all $B \in \BB'$. Covering $B(w,r)$ by $M$ balls of radius $cr$, we see that at least one of
the balls, say $B(y,cr)$, has a $(\Gamma^qc\delta r)$-packing $\BB \subset \BB'$ so that
\begin{equation}\label{cc}
\#\BB \ge \delta^{-t}/M^2.
\end{equation}
Since $B(z,\Lambda\Gamma^qcr) \subset B(w,5\Gamma r)$ for $z\in B(y,2cr)$ we now have
  \begin{equation*}
    \hom^\Lambda_{\delta,\eps,\Gamma^qcr}(\mu,z) > \delta^{-t}/M^2
  \end{equation*}
  for all $z \in B(y,2cr)$, and, consequently, $A \cap B(y,2cr) = \emptyset$. Using \eqref{aa}--\eqref{cc}, we estimate
  \begin{equation*}
    \mu\bigl( B(w,\Gamma r) \setminus A \bigr) < \#\BB \eps\mu\bigl( B(w,5\Gamma r) \bigr) \le \sum_{B \in \BB}\mu(B) \le \mu\bigl( B(y,2cr) \bigr)=\mu\bigl(B(y,2cr)\setminus A \bigr).
  \end{equation*}
Since $B(y,2cr) \subset B(w,\Gamma r)$ we arrive at a contradiction.
\end{proof}

\begin{remark}
In general, the equality of Lemma \ref{prop:5} might not hold at every point $x \in X$ even when $X = \R^2$. To see this take
\begin{equation*}
\mu = \sum_{k=1}^\infty \frac{1}{k!}\HH^1|_{S^1(0,2^{-k})},
\end{equation*}
where $\HH^1|_{S^1(0,2^{-k})}$ is the length measure on $S^1(0,2^{-k}) = \{ y \in \R^2 : |y| = 2^{-k} \}$.
Then we have $\dim_{\hom}^{3/2}\bigl( \mu,(0,0) \bigr) = 1$, but $\dim_{\hom}^{5/2}\bigl( \mu,(0,0) \bigr) = 0$.
\end{remark}

\section{Main results} \label{sec:main_results}

\subsection{Relating $L^q$-dimensions with local dimensions}\label{sec:Lqresults}
The $L^q$-spectrum of a measure is an essential tool in multifractal analysis and it has been
investigated in many works, see e.g.\ \cite{Ngai1997, Heurteaux1998, LauNgai1999, PeresSolomyak2000, BarbarouxGerminetTcheremchantsev2001, FanLauRao2002, Shmerkin2005b}
and \cite{Falconer1990, Falconer1997, Pesin1997} and references therein. It turns out that the well known Hausdorff
and packing dimension estimates for the measure arising from its global $L^q$-spectrum generalise to the setting of
local spectrum in doubling metric spaces. Compare the following result to \cite[Theorem 1.3]{Heurteaux1998},
\cite[Theorem 1.1]{Ngai1997}, and \cite[Theorem 1.4]{FanLauRao2002}. See also \cite[Corollary 1.3]{Olsen2000}.

\begin{theorem} \label{thm:pointwise_dimq}
  If $\mu$ is a measure on a doubling metric space $X$, then
  \begin{equation} \label{eq:local_dimq}
    \lim_{q \downarrow 1} \dim_q(\mu,x) \le \ldimloc(\mu,x) \le \udimloc(\mu,x) \le \lim_{q \uparrow 1} \dim_q(\mu,x)
  \end{equation}
  for $\mu$-almost all $x \in X$.
\end{theorem}

The proof of Theorem \ref{thm:pointwise_dimq} is postponed until the end of this section.
We remark that all the inequalities in \eqref{eq:local_dimq} can be strict; see e.g.\ Remark \ref{rem:strict}.

\begin{lemma}\label{lemma:globloc}
  If $X$ is a doubling metric space, $A\subset X$ bounded, $r>0$, $\mu$ a measure on $X$, $q\in\R$ and $0<\delta<r$,
 then there is an $r$-packing $\BB$ of $A$ so that \[S_q(\mu,A,\delta)\le c\sum_{B \in \BB} S_q(\mu,B,\delta)\,,\]
where $c=c(N)\in\N$.
\end{lemma}

\begin{proof}
Using Lemma \ref{thm:covering_thm}\eqref{covering5}, we choose $r$-packings of $A$, say $\BB_1,\ldots,\BB_M$ where $M=M(N) \in \N$, whose union covers $A$.
Fix $0<\delta<r$ and let $\mathcal{B}'$ be a $\delta$-packing of $A\cap\spt(\mu)$ such that $2\sum_{B\in\mathcal{B}'}\mu(B)^q>S_q(\mu,A,\delta)$. If $\BB'_B = \{ B' \in \BB' : \text{the center point of $B'$ is in $B$} \}$ for all $B \in \bigcup_{i=1}^M \BB_i$, then
\[\sum_{B'\in\mathcal{B}'}\mu(B')^q \le \sum_{i=1}^M\sum_{B \in \BB_i}\sum_{B' \in \BB'_B}\mu(B')^q \le \sum_{i=1}^M\sum_{B \in \BB_i} S_q\bigl( \mu,B,\delta \bigr).\]
Thus $2M \sum_{B \in \BB_i} S_q\bigl( \mu,B,\delta \bigr)\ge S_q(\mu,A,\delta)$ for some $i$.
\end{proof}

\begin{lemma}\label{lemma:taudec}
If $\mu$ is a measure on a doubling metric space $X$, then for any $q\ge 0$ and $\varepsilon>0$,
there is a countable covering of $X$ by bounded sets $A$ for which $\sup_{x\in A}\tau_q(\mu,x) \le \tau_q(\mu,A)+\varepsilon$.
\end{lemma}

\begin{proof}
We may cover $X$ by countably many sets of the form
\[A_\alpha=\{x\in X\,:\, \alpha<\tau_q(\mu,x)<\alpha+\varepsilon\}.\]
If $x\in A_\alpha$, then there exist $r>0$ and $\delta_0>0$ such that $S_q\bigl( \mu,B(x,r),\delta \bigr) < \delta^\alpha$
for all $0<\delta<\delta_0$. Thus, $A_\alpha$ can be covered by countably many sets of the form
\[A_{\alpha,r,\delta_0,R}=\{x\in A_\alpha\cap B(x_0,R) : S_q\bigl( \mu,B(x,r),\delta\bigr) < \delta^\alpha\text{ for all }0<\delta<\delta_0\}.\]
By Lemma \ref{lemma:globloc}, we find an $r$-packing $\mathcal{B}$ of $A_{\alpha,r,\delta_0,R}$ so that
\begin{align*}
\frac{\log S_q(\mu,A_{\alpha,r_0,\delta_0,R},\delta)}{\log\delta} \ge \frac{\log c\sum_{B \in \BB} S_q\bigl(\mu, B,\delta\bigr)}{\log\delta}\ge\frac{\log(\#\BB c\delta^\alpha)}{\log\delta},
\end{align*}
where $c=c(N) \in \N$. Since $\BB$ has at most $M=M(r,R,N) \in \N$ elements by Lemma
\ref{thm:covering_thm}\eqref{covering3}, we get $\tau_q(\mu,A_{\alpha,r,\delta_0,R})\ge\alpha$ by letting $\delta\downarrow 0$.
\end{proof}

The following lemma can be considered as a global version of Theorem \ref{thm:pointwise_dimq}.

\begin{lemma}\label{2-3}
If $\mu$ is a measure on a doubling metric space $X$ and $A\subset X$ is bounded, then
\begin{align*}
\dim_q(\mu,A) &\le \mu\text{-}\essinf\{ \ldimloc(\mu,x) : x \in A \} \\
&\le \mu\text{-}\esssup\{ \udimloc(\mu,x) : x \in A \} \le \dim_p(\mu,A)
\end{align*}
for all $0<p<1<q$.
\end{lemma}

\begin{proof}
Let $q>1$. If $s > \mu\text{-}\essinf\{ \ldimloc(\mu,x) : x \in A \}$ and $A_n = \{ x \in A \cap \spt(\mu) : \mu\bigl( B(x,2^{-n}) \bigr) > 2^{-ns} \}$, then $\sum_{n=1}^\infty\mu(A_n)=\infty$ by the Borel-Cantelli Lemma. Thus, there are arbitrarily large $n$ such that $\mu(A_n)>n^{-2}$. Fix such an $n$ and let $\BB$ be a $(2^{-n})$-packing of $A_n$ satisfying \eqref{packingestimate}. Then
\begin{align*}\label{Asum}
S_{q}(\mu,A,2^{-n})&\ge\sum_{B \in \BB}\mu(B)^q=\sum_{B \in \BB}\mu(B)\mu(B)^{q-1} \\ &\ge\sum_{B \in \BB}\mu(B)2^{-ns(q-1)}
\ge c\mu(A_n) 2^{-ns(q-1)}\ge c n^{-2} 2^{-ns(q-1)}.
\end{align*}
Taking logarithms and letting $n\rightarrow\infty$, this implies $\tau_q(\mu,A) \le s(q-1)$ and,
consequently, $\dim_q(\mu,A)\le s$ as required.

If $0<p<1$, then we complete the proof by repeating the above argument with $q$ replaced by $p$,
$s < \mu\text{-}\esssup\{ \udimloc(\mu,x) : x \in A \}$, and $A_n = \{ x \in A \cap \spt(\mu) : \mu\bigl( B(x,2^{-n}) \bigr) < 2^{-ns} \}$.
\end{proof}

\begin{proof}[Proof of Theorem \ref{thm:pointwise_dimq}]
The proof follows simply by combining Lemmas \ref{lemma:taudec} and \ref{2-3}. Indeed,
for $q>1$ and $\varepsilon>0$, decompose $X$ into countably many bounded sets $A$ for
which $\sup_{x\in A}\tau_q(\mu,x) \le \tau_q(\mu,A)+\varepsilon$. Lemma \ref{2-3} then implies that
\begin{equation*}
\dim_q(\mu,x)-\varepsilon/(q-1) \le \dim_q(\mu,A) \le \ldimloc(\mu,x)
\end{equation*}
for $\mu$-almost all $x\in A$. The leftmost inequality of \eqref{eq:local_dimq} follows
now by letting $\eps \downarrow 0$. Recall that the limit exists by Lemma \ref{tauprop}(5).
The proof in the case $0<q<1$ is similar.
\end{proof}

\subsection{Upper bound by local homogeneity dimension} \label{sec:local_homo}
In this section, we prove our main result showing that the local homogeneity dimension
is almost everywhere at least as large as the upper local dimension.

\begin{theorem} \label{thm:homodim}
   If $\mu$ is a measure on a doubling metric space $X$ and $\Lambda > 1$, then
   \begin{equation*}
     \udimloc(\mu,x) \le \dim_{\hom}^\Lambda(\mu,x)
   \end{equation*}
   for $\mu$-almost all $x \in X$.
\end{theorem}

Theorem \ref{thm:homodim} is obtained as a corollary to a more
quantitative result, Theorem \ref{thm:main}, which will be essential
in our applications in \S \ref{sec:applications}.
Before we turn to Theorem \ref{thm:main}, we exhibit some auxiliary results.
We first observe that the homogeneity can be used to bound the $L^q$-moment sums.

\begin{lemma}\label{lemma:homtoS}
If $X$ is a doubling metric space and $0<\delta <1$, then there is $M=M(\delta,N) \in \N$ so that for every measure
 $\mu$ on $X$ and for all $\Lambda > 1$, $0<q<1$, $x \in X$, and $r,\eps > 0$ we have
\[S_{q}\bigl( \mu,B(x,r),\delta r \bigr) \le \bigl( \hom_{\delta,\varepsilon,r}^\Lambda(\mu,x)^{1-q}+ M\varepsilon^q \bigr)\mu\bigl( B(x,\Lambda r) \bigr)^q.\]
\end{lemma}

\begin{proof}
If $\mathcal{B}$ is a $(\delta r)$-packing of $B(x,r)$ and $\mathcal{B}' = \{ B\in\mathcal{B} : \mu(B)>\varepsilon\mu\bigl( B(x,\Lambda r) \bigr) \}$,
then H\"older's inequality implies
\begin{equation*}
\sum_{B\in\mathcal{B}'}\mu(B)^q \le \hom_{\delta,\varepsilon,r}^\Lambda(\mu,x)^{1-q}\mu\bigl( B(x,\Lambda r) \bigr)^q.
\end{equation*}
On the other hand, since
\begin{equation*}
\sum_{B\in\mathcal{B}\setminus\mathcal{B}'} \mu(B)^q \le \#\mathcal{B}\varepsilon^q\mu\bigl( B(x,\Lambda r) \bigr)^q
\end{equation*}
and $\#\mathcal{B}\le M(\delta,N)$ by Lemma \ref{thm:covering_thm}\eqref{covering3}, the claim follows.
\end{proof}

\begin{lemma}\label{lemma:joku}
If $X$ is a doubling metric space, $0<q,\delta<1$, and $m>0$, then there exists a constant
$\varepsilon=\varepsilon(q,\delta,m,N)>0$ satisfying the following: If $\mu$ is a measure on $X$, $\Lambda>1$,
and $A\subset X$ is bounded so that $\hom_{\delta,\varepsilon,r}^\Lambda(\mu,x)\le\delta^{-m}$ for all $x\in A$
and $0<r<r_0$, then there is a constant $c=c(N,\Lambda) \ge 1$ so that
\[S_q(\mu,A,\delta r)\le c\delta^{m(q-1)}S_q(\mu,A,\Lambda r)\]
for all $0<r<r_0$.
\end{lemma}

\begin{proof}
Let $\varepsilon>0$ be so small that $M\varepsilon^q \le \delta^{m(q-1)}$, where $M$ is as in Lemma \ref{lemma:homtoS}.
According to Lemma \ref{lemma:globloc}, there are $c=c(N)\in\N$ and an $r$-packing $\BB$ of $A$ so that
\begin{equation*}
  S_q(\mu,A,\delta r) \le c\sum_{B \in \BB} S_q(\mu,B,\delta r)
  \le c\sum_{B \in \BB} 2\delta^{m(q-1)} \mu(\Lambda B)^q
\end{equation*}
by Lemma \ref{lemma:homtoS}, the homogeneity assumption and the choice of $\eps$.
The claim now follows since $\sum_{B\in\mathcal{B}}\mu(\Lambda B)^q\le c(N,\Lambda) S_q(\mu,A,\Lambda r)$
by Lemma \ref{thm:covering_thm}\eqref{covering4}.
\end{proof}

The following theorem is our main quantitative result concerning
local homogeneity of measures.

\begin{theorem}\label{thm:main}
  If $X$ is a doubling metric space, $0 < m < s$, and $\Lambda>1$, then there exists a constant
$\delta_0 = \delta_0(m,s,N,\Lambda) > 0$ satisfying the following: For every $0 < \delta < \delta_0$ there is
$\eps_0 = \eps_0(\delta,m,N) > 0$ so that for each
  measure $\mu$ on $X$ we have
\begin{equation}\label{hommu}
\limsup_{r \downarrow 0}
  \hom_{\delta,\eps,r}^\Lambda(\mu,x) \geq \delta^{-m}
\end{equation}
for all $0 < \eps \leq \eps_0$ and for $\mu$-almost all $x \in X$ that satisfy $\udimloc(\mu,x) > s$.
\end{theorem}

\begin{proof}
Fix $0<q<1$. Let $\delta_0>0$ be so small that $\log (c\Lambda^{m(q-1)})/\log((\delta_0/\Lambda)^{q-1}) > (s-m)(q-1)$,
where $c=c(N,\Lambda)>0$ is as in Lemma \ref{lemma:joku}. Fix $0<\delta<\delta_0$ and let $\eps = \eps(q,\delta,m,N) > 0$
be as in Lemma \ref{lemma:joku}. Given $x_0 \in X$ and $R,r_0 > 0$, it suffices to show that $\udimloc(\mu,x) \le s$ for
$\mu$-almost every point in the set
\begin{equation*}
  A = \{x\in B(x_0,R) : \hom_{\delta,\varepsilon,r}^\Lambda(\mu,x) < \delta^{-m} \text{ for all } 0<r<r_0\}.
\end{equation*}
According to Lemma \ref{lemma:joku}, we have $S_q(\mu,A,\delta r/\Lambda) \le c\delta^{m(q-1)} S_q(\mu,A,r)$
for all $0<r<r_0$. A simple induction gives $S_q\bigl( \mu,A,(\delta/\Lambda)^nr_0 \bigr) \le c^n\delta^{nm(q-1)} S_q(\mu,A,r_0)$
for all $n \in \N$. Therefore
\begin{equation*}
  \tau_q(\mu,A) = \liminf_{n \to \infty} \frac{\log S_q\bigl( \mu,A,(\delta/\Lambda)^nr_0 \bigr)}{\log\bigl( (\delta/\Lambda)^nr_0 \bigr)} \ge m(q-1) + \frac{\log (c\Lambda^{m(q-1)})}{\log(\delta/\Lambda)}
\end{equation*}
and so $\dim_q(\mu,A) \le s$ by the choice of $\delta_0$. Lemma \ref{2-3} then gives
 $\udimloc(\mu,x)\le s$ at $\mu$-almost all points $x\in A$.
\end{proof}

\begin{remark}
In general, it is possible that $\dim_q(\mu,x)>c>0$ for all $0<q<1$ almost everywhere
while $\dim_{\hom}(\mu,x)=0$; see Example \ref{ex:q>h}. It is essential in the proof of
Theorem \ref{thm:main} that in the set $A$, where we have uniform estimates for
$\hom_{\delta,\varepsilon,r}^\Lambda(\mu,x)$, we can use $\dim_{\hom}(\mu,x)$ to bound $\dim_q(\mu,A)$ from above.
\end{remark}

\begin{proof}[Proof of Theorem \ref{thm:homodim}]
  Assume to the contrary that there are $A\subset X$ with $\mu(A)>0$ and $0<m<s$
such that $\dim_{\hom}^\Lambda(\mu,x)<m<s<\udimloc(\mu,x)$ for all $x \in A$. It
follows from Theorem \ref{thm:main} that there is $\delta_0=\delta_0(m,s,N,\Lambda)>0$
so that $\hom^{\Lambda}_\delta(\mu,x) \ge \delta^{-m}$ for every $0<\delta<\delta_0$
and for $\mu$-almost all $x \in A$. Thus $\dim_{\hom}^\Lambda(\mu,x)\geq m$ for $\mu$-almost
all $x\in A$ giving a contradiction.
\end{proof}

\subsection{Entropy dimension}\label{sec:entropydim}

We complete the discussion on $\dim_q$ by treating the case $q=1$.
This is done by defining for $A\subset X$ with $\mu(A)>0$ the \emph{(global) upper and lower entropy dimensions} of $\mu$ on $A$ as
\begin{equation*}
\begin{split}
  \udim_1(\mu,A) &= \limsup_{\delta \downarrow 0} \fint_{A} \frac{\log\mu\bigl( B(y,\delta) \bigr)}{\log\delta}\,d\mu(y), \\
  \ldim_1(\mu,A) &= \liminf_{\delta \downarrow 0} \fint_{A} \frac{\log\mu\bigl( B(y,\delta) \bigr)}{\log\delta}\,d\mu(y),
\end{split}
\end{equation*}
respectively. If they agree, then their common value is denoted by $\dim_1(\mu,A)$.
Here and hereafter, for $A\subset X$ and a $\mu$-measurable $f\colon
X\rightarrow\overline{\R}$, we use notation $\fint_A f(y)\,d\mu(y) =
\mu(A)^{-1} \int_A f(y)\,d\mu(y)$ whenever the integral is well
defined. The \emph{local upper and lower entropy dimensions at $x\in\spt(\mu)$} are then defined as
\begin{equation} \label{eq:dim1}
\begin{split}
  \udim_1(\mu,x) &= \limsup_{r \downarrow 0} \udim_1\bigl( \mu,B(x,r) \bigr),\\
  \ldim_1(\mu,x) &= \liminf_{r \downarrow 0} \ldim_1\bigl( \mu, B(x,r) \bigr).
\end{split}
\end{equation}

Our results on $\dim_1(\mu,x)$ are local metric space versions of the corresponding
global Euclidean results. For instance, see \cite[Theorem 4.1]{Heurteaux1998} and \cite[Theorem 1.4]{FanLauRao2002}.
The case $q=1$ is different from $q\neq 1$ in the sense that it cannot be studied
solely by using Borel-Cantelli type arguments. Also, in the main result of this section,
Theorem \ref{thm:entropy}, the density point property is a crucial assumption and it cannot
be replaced by the weaker and more general condition given by Lemma \ref{lemma:weakdpp} as
is the case for $q\neq 1$.

The following proposition shows that the definition of $\dim_1$ is consistent with the basic properties of $\dim_q$.

\begin{proposition}\label{prop:dim1}
If $\mu$ is a measure on a doubling metric space $X$, then
\begin{equation*}
\lim_{q \downarrow 1} \dim_q(\mu,x) \le \ldim_1(\mu,x) \le \udim_1(\mu,x) \le \lim_{q \uparrow 1} \dim_q(\mu,x)
\end{equation*}
for all $x\in\spt(\mu)$.
\end{proposition}

The proof of the proposition involves the partition definition of $\dim_q$.
Since we do not need the result in this article, we will omit the proof.
A detailed proof can be found in \cite{KaenmakiRajalaSuomala2012b}.

\begin{theorem}\label{thm:entropy}
If $\mu$ is a measure on a doubling metric space $X$ so that it satisfies the density point property, then
  \begin{equation} \label{eq:local_dim1}
    \ldimloc(\mu,x) \le \ldim_1(\mu,x) \le \udim_1(\mu,x) \le \udimloc(\mu,x)
  \end{equation}
  for $\mu$-almost all $x \in X$.
\end{theorem}

\begin{proof}
We may assume that the measure is non-atomic as the claim is obvious if $\mu(\{x\})>0$.
Given $\varepsilon>0$, we may cover $\mu$-almost all of $X$
by countably many sets of the form
$A'=\{y\in X : t<\ldimloc(\mu,y)<t+\varepsilon\}$ and each of these can be covered by countably many sets of the form
$A=\{y\in A' : \mu\bigl( B(x,r) \bigr) < r^t \text{ for all } 0<r<q\}$. For  $x\in\spt(\mu)$ and $0<\delta<q$, we have
\begin{align*}
\fint_{B(x,r)}\frac{\log\mu\bigl(B(y,\delta)\bigr)}{\log\delta}\,d\mu(y)&\ge\frac1{\mu\bigl(B(x,r)\bigr)}\int_{A\cap B(x,r)}\frac{\log\mu\bigl(B(y,\delta)\bigr)}{\log\delta}\,d\mu(y)\\
&\ge t\frac{\mu\bigl(A \cap B(x,r)\bigr)}{\mu\bigl(B(x,r)\bigr)}.
\end{align*}
Here we have assumed that $\mu\bigl(B(x,r+\delta)\bigr)<1$. For small $r$ and $\delta$ this is the case since $\mu$ has no atoms. Since almost all points $x\in A$ are density points, we get
\[\ldim_1(\mu,x)=\liminf_{r\downarrow 0}\liminf_{\delta\downarrow 0}\fint_{B(x,r)}\frac{\log\mu\bigl(B(y,\delta)\bigr)}{\log\delta}\,d\mu(y) \ge t\]
for $\mu$-almost all $x\in A$
and consequently $\ldim_1(\mu,x)\ge\ldimloc(\mu,x)-\varepsilon$ for $\mu$-almost all $x\in X$.

To prove the estimates for the upper dimension, a similar covering argument as above implies that it suffices to show that if $0<q,t<\infty$, then $\udim_1(\mu,x)\le t$ for $\mu$-almost all $x\in A = \{y\in X : \mu\bigl( B(y,r) \bigr) > r^t \text{ for all } 0<r<q\}$. Let $x\in X$ and $0<r<q$. For $0<\delta<q$ and $t\le\alpha<\infty$, define $E_{\delta,\alpha}=\{y\in B(x,r) : \mu\bigl( B(y,\delta) \bigr) < \delta^\alpha\}$. By Lemma \ref{thm:covering_thm}\eqref{covering2}, $E_{\delta,\alpha}$ can be covered by $C\delta^{-s}$ balls of radius $\delta$ with centres in $E_{\delta,\alpha}$, where $s=s(N)>0$ and $C=C(N,r)>0$. Thus $\mu(E_{\delta,\alpha}) \le C\delta^{\alpha-s}$.
Let $J = \{ y \in B(x,r) : \mu\bigl( B(y,\delta) \bigr) > \delta^t \}$, $K = \{ y \in B(x,r) : \delta^{2s} \le \mu\bigl( B(y,\delta) \bigr) \le \delta^{t} \}$, and $L = \{ y \in B(x,r) : \mu\bigl( B(y,\delta) \bigr) < \delta^{2s} \}$. Then $B(x,r) = J \cup K \cup L$. Moreover,
\begin{align*}
\int_{J}\frac{\log\mu\bigl( B(y,\delta) \bigr)}{\log\delta} \,d\mu(y) &\le t\mu(J) \le t\mu\bigl( B(x,r) \bigr), \\
\int_{K}\frac{\log\mu\bigl( B(y,\delta) \bigr)}{\log\delta} \,d\mu(y) &\le 2s\mu(K) \le 2s\mu\bigl( B(x,r)\setminus A \bigr), \\
\int_{L}\frac{\log\mu\bigl( B(y,\delta) \bigr)}{\log\delta} \,d\mu(y) &= \int_{2s}^\infty \mu(E_{\delta,\alpha}) \,d\alpha \le C\delta^{-s} \int_{2s}^\infty \delta^{\alpha}\,d\alpha = \frac{C\delta^s}{-\log\delta}.
\end{align*}
Putting these together and letting $\delta\downarrow 0$ in the last estimate, we get
\[\udim_1\bigl( \mu,B(x,r) \bigr) = \limsup_{\delta\downarrow 0}\fint_{B(x,r)}\frac{\log\mu\bigl( B(y,\delta) \bigr)}{\log\delta}\,d\mu(y) \le t+2s\frac{\mu\bigl( B(x,r) \setminus A \bigr)}{\mu\bigl( B(x,r) \bigr)}\]
and, consequently,
\[\udim_1(\mu,x)=\limsup_{r\downarrow 0}\udim_1\bigl( \mu,B(x,r) \bigr) \le t\]
for all density points of $A$. The claim follows since $\mu$ has the density point property.
\end{proof}

\begin{remark}\label{rem:dppneeded}
(1) By inspecting the above proof, we easily get a global analogue of Theorem \ref{thm:entropy}:
If $A\subset X$ is bounded and $\mu(A) > 0$, then it holds that
\begin{align*}
\mu\text{-}\esssup_{x\in A}\udimloc(\mu,x)&\ge \udim_1(\mu,A), \\
\mu\text{-}\essinf_{x\in A}\ldimloc(\mu,x)&\le \ldim_1(\mu,A).
\end{align*}
It is worthwhile to notice that the density point property is not needed in this case.

(2) In Examples \ref{ex:nodpp1} and \ref{ex:nodpp2}, we show that Theorem \ref{thm:entropy}
 does not hold without the density point property. This is a remarkable difference between
the global and local entropy dimensions.
\end{remark}

\section{Applications} \label{sec:applications}

In this section, we use the local homogeneity estimate of Theorem
\ref{thm:main} as the final step in proving various new results.
In fact, understanding the conical density and porosity questions in
\S \ref{conical}--\S \ref{metricporo} below was our main motivation for
investigating the local homogeneity. In addition to Theorem
\ref{thm:main}, the proofs will be based on already known geometric
conclusions.

\subsection{Upper conical densities in Euclidean spaces}\label{conical}

Let $G(d,n)$ be the Grasmann manifold of all $n$-dimensional linear subspaces
of $\R^d$ and $S^{d-1} = \{y \in \R^d : |y| =1\}$ the unit sphere in $\R^d$.
Then for $0 < \alpha \le 1$, $V \in G(d,d-k)$, $\theta \in S^{d-1}$, $x \in \R^d$
and $r>0$ we define cones
\[
 X(x,r,V,\alpha) = \{y \in B(x,r) : \dist(y-x,V)< \alpha|y-x|\}
\]
and
\[
 H(x,\theta,\alpha) = \{y \in \R^d : (y-x)\cdot \theta > \alpha|y-x|\}.
\]
With small $\alpha$ the cones $X(x,r,V,\alpha)$ are small cones around
the translate of the subspace $V$ by $x$, whereas the cone $H(x,\theta,\alpha)$
is almost a half-space from the point $x$ to the direction $\theta$.

The distribution of Hausdorff and packing type measures inside cones
is well studied and understood, see for example \cite{Marstrand1954,
  Salli1985, Mattila1988, Suomala2005a, KaenmakiSuomala2008}. For general measures
the following theorem was proved in \cite[Theorem
4.1]{CsornyeiKaenmakiRajalaSuomala2010} under the assumption that
the Hausdorff dimension of the measure is greater than $s$. We improve
this result by showing that the theorem
is true even if we assume a lower bound only for the packing (i.e.\ the upper local)
dimension of the measure.

\begin{theorem}\label{thm:conical}
  If $d \in \N$, $k \in \{ 0 ,\ldots,d-1\}$, $s>k$, and $0 < \alpha
  \le 1$, then there exists a constant $c = c(d,k,s,\alpha) > 0$ so
  that for every measure $\mu$ on $\R^d$
  we have
  \begin{equation*}
    \limsup_{r \downarrow 0} \inf_{\atop{\theta\in S^{d-1}}{V \in
        G(d,d-k)}} \frac{\mu\bigl( X(x,r,V,\alpha) \setminus
      H(x,\theta,\alpha) \bigr)}{\mu\bigl( B(x,r) \bigr)} > c
  \end{equation*}
  for $\mu$-almost all $x \in \R^d$ that satisfy $\udimloc(\mu,x)>s$.
\end{theorem}
\begin{proof}
We can reduce the proof to verifying the following condition, see
\cite[Proposition 4.5]{CsornyeiKaenmakiRajalaSuomala2010}:
For a given
$q,K \in \N$ and $1 < t < \infty$ there exists a constant $\eps =
\eps(d,k,s,q,K,t)>0$ so that for $\mu$-almost all $x
\in \{ y \in \R^d : \udimloc(\mu,y)>s \}$ we may
find arbitrarily small radii $r>0$ and ball families $\BB$ with the
following properties:
\begin{enumerate}
  \item \label{assu1} $B \subset B(x,r)$ for all $B \in \BB$.
  \item \label{assu2} The collection $t\BB = \{tB : B\in\mathcal{B}\}$
    is a packing.
  \item \label{assu3} $\mu(B) > \eps\mu\bigl( B(x,3r) \bigr)$  for all
    $B\in\mathcal{B}$.
  \item \label{assu4} If $\mathcal{B}'\subset\mathcal{B}$ with
    $\#\mathcal{B}'\geq\#\mathcal{B}/K$ and $V\in G(d,d-k)$, then
    there is
a translate of $V$ intersecting at least $q$ balls from the collection
$\mathcal{B'}$.
\end{enumerate}
We will construct the families $\BB$ with the help of Theorem
\ref{thm:main}. Let $M = M(N_d,t^{-1})$ be the constant from Lemma
\ref{thm:covering_thm}\eqref{covering4}, where $N_d$ is the doubling constant of
$\R^d$. Let $m = (s+k)/2$ and choose
$0<\delta<\min\{\delta_0,\frac{1}{4}\}$ so that $4^{-k}\delta^{k-m}
\ge 2KMq$, where $\delta_0$ is as in Theorem \ref{thm:main}. By
Theorem \ref{thm:main} there is $\eps = \eps(m,s,N_d,\delta)> 0$ so that
$\limsup_{r \downarrow 0} \hom^{5}_{\delta,\eps,r}(\mu,x)\geq\delta^{-m}$ for
$\mu$-almost all $x \in \R^d$ that satisfy $\udimloc(\mu,x)>s$.
Fix such a point $x$ and let $r>0$ so
that $\hom_{\delta,\eps,\frac{3}{4}r}(\mu,x)> \delta^{-m}/2$. Now there
is a $(\frac{3}{4}\delta r)$-packing of $B(x,\frac{3}{4}r)$, say
$\BB_0$, with $\#\BB_0 > \delta^{-m}/2$ so that $\mu(B) > \eps\mu\bigl(
B(x,\frac{15}{4}r) \bigr)\geq \eps\mu\bigl( B(x,3r) \bigr)$ for all $B \in \BB_0$.

Lemma \ref{thm:covering_thm}\eqref{covering4} gives a subcollection $\BB \subset
\BB_0$ for which $t\BB$ is also a packing and $\#\BB \ge \#\BB/M \ge
\delta^{-m}/(2M)$. Now, because $\delta \le \frac{1}{4}$, $B \subset
B(x,r)$ for each $B \in \BB$. Thus conditions
(\ref{assu1})--(\ref{assu3}) hold. The only property we need to verify
is the condition (\ref{assu4}). Suppose that $\BB' \subset \BB$ with
$\#\BB' \ge \#\BB/K \ge \delta^{-m}/(2KM)$, and let $V \in
G(d,d-k)$. The orthogonal projection of $B(x,r)$ into the orthogonal
complement of $V$ can be covered by $4^k\delta^{-k}$ balls of radius
$\frac{3}{4}\delta r$ and so some translate of $V$ must intersect at
least
\[
4^{-k}\delta^k\#\BB' \ge \frac{4^{-k}\delta^{k-m}}{2KM}\ge q
\]
balls from the collection $\BB'$. Thus also (\ref{assu4}) holds and the proof is finished.
\end{proof}

\subsection{Porous measures on Euclidean spaces}\label{sec:kpor}
We first define porosity for sets. Let $A \subset \R^d$, $k \in \{ 1,\ldots,d \}$, $x \in A$, and $r > 0$. We define
\[
 \begin{split}
  \por_k(A,x,r) = \sup\{\roo \ge 0 :\;&\text{there are }y_1,\dots, y_k\in \R^d \text{ such that for every } i \\
  &A \cap B(y_i,\roo r) = \emptyset \text{ and } \roo r+|x-y_i| \leq r,\\
  &\text{and }(y_i-x)\cdot(y_j-x) = 0 \text{ if } j \ne i\}
\end{split}
\]
and from this the \emph{$k$-porosity of $A$ at $x$} as
\[
 \por_k(A, x) = \liminf_{r \downarrow 0} \por_k(A, x, r).
\]
We refer to the balls $B(y_i,\roo r)$ in the definition as holes.
The notion of $k$-porosity was introduced in \cite{KaenmakiSuomala2004}.

When we combine this definition with the porosity for measures,
defined for the first time in
\cite{EckmannJarvenpaaJarvenpaa2000}, we obtain $k$-porosity for
measures: Let $\mu$ be a measure on $\R^d$, $k \in \{ 1,\ldots,d \}$,
$x \in \R^d$, $r>0$, and $\eps > 0$. We set
\[
 \begin{split}
  \por_k(\mu,x,r,\eps) = \sup\{\roo \ge 0 :\;&\text{there are }y_1,\dots, y_k\in \R^d \text{ such that for every } i \\
  &\mu\bigl( B(y_i,\roo r) \bigr) \le \eps \mu\bigl( B(x,r) \bigr) \text{ and }\roo r+|x-y_i| \leq r, \\
  &\text{and }(y_i-x)\cdot(y_j-x) = 0 \text{ if } j \ne i\}
\end{split}
\]
and the \emph{$k$-porosity of the measure $\mu$ at $x$} is defined to be
\[
 \por_k(\mu, x) = \lim_{\eps \downarrow 0}\liminf_{r \downarrow 0}\por_k(\mu,x,r,\eps).
\]
It follows from \cite[\S 2]{EckmannJarvenpaaJarvenpaa2000} that $\por_k(\mu,x) \le \tfrac12$
for $\mu$-almost all $x \in \R^d$.
We remark that a more precise name for the porosity just defined would be lower porosity,
to distinguish this notion from the upper porosity of sets and measures, see e.g.\ \cite{MeraMoranPreissZajicek2003, Suomala2008}.

We provide an upper bound for the upper local dimension of measures
with $k$-porosity close to the maximum value $\tfrac12$. In
\cite{RajalaSmirnov2007}, this result was proved for $k=1$. The first
estimates for the dimension of sets with $1$-porosity close to
$\frac{1}{2}$ are from \cite{Mattila1988} and \cite{Salli1991}. For
more recent results on the dimension of porous sets and measures; see
\cite{KaenmakiSuomala2004, JarvenpaaJarvenpaaKaenmakiSuomala2005, Rajala2009, Chousionis2008b} and
\cite{EckmannJarvenpaaJarvenpaa2000, JarvenpaaJarvenpaa2002, BeliaevSmirnov2002, KaenmakiSuomala2008, RajalaSmirnov2007}.
It is important to notice both here and in Theorem
\ref{thm:metricporo} that even if $\por_1(\mu,x)>0$ in a set of positive $\mu$-measure,
it is possible that $\mu(A)=0$ for all $A \subset X$ with $\inf_{x \in A}\por_1(A,x)>0$,
see \cite[Theorem 4.1]{RajalaSmirnov2007}.

\begin{theorem}\label{thm:kpor}
  If $d \in \N$, then there exists a constant $c=c(d)>0$ so that for every measure $\mu$ on $\R^d$ we have
  \begin{equation*}
    \udimloc(\mu,x) \le d-k+\frac{c}{-\log\bigl( 1-2\por_k(\mu,x) \bigr)}
  \end{equation*}
  for $\mu$-almost all $x\in\R^d$.
\end{theorem}

\begin{remark} \label{pororem}
  (1) It is rather easy to see that the upper bound in Theorem
  \ref{thm:kpor} is asymptotically sharp as
  $\por_k(\mu,x)\uparrow\tfrac12$: For each $\roo<\tfrac12$ there
  exists a measure $\mu$ on $\R^d$ with $\por_k(\mu,x)\ge\roo$ while
  $\ldimloc(\mu,x)\ge d-k-c/\log(1-2\roo)$ for $\mu$-almost all $x \in
  \R^d$. The easiest way to see this is to consider a regular Cantor
  set $C\subset\R$ with
  $1$-porosity $\roo$ and to let $\mu$ be the natural measure on
  $C^k\times[0,1]^{d-k}$.

  (2) The proof of Theorem \ref{thm:kpor} in the case $k=1$ given in
  \cite{RajalaSmirnov2007} is based on an extensive use of dyadic cubes.
  The interplay
  between cubes and balls caused many technical problems, which were
  finally solved by considering the boundary regions of cubes
  separately. The method used there does not work for $k$-porosity
  when $k \ge 2$ although the statement itself has nothing to do with
  co-dimension being one.
\end{remark}

Before proving Theorem \ref{thm:kpor}, we will exhibit a couple of geometric lemmas concerning $k$-porous sets.

\begin{lemma}\label{lma:kporcover}
If $A \subset B(x_0,r) \subset \R^d$ is so that $\por_k(A,z,r) \ge
\roo$ for every $z \in A$, then the set $A$ can be covered with
$c(1-2\roo)^{k-d}$
balls of radius $(1-2\roo)r$, where $c > 0$ depends only on $d$.
\end{lemma}

\begin{proof}
  The proof is based on similar geometric arguments as used in
  \cite[Theorem 2.5]{JarvenpaaJarvenpaaKaenmakiSuomala2005},
  \cite[Lemmas 3.4 and 3.5]{RajalaSmirnov2007}, and \cite[Lemma
  5.1]{Rajala2009}.
  In the proof, we will omit some of the elementary, if tedious, details.

  Let $c_1,c_2,c_3>0$ be small constants.
  We may assume that $\roo>\tfrac12-c_1$. A simple compactness argument
  implies that $\R^d$ can be covered by $m=m(d,c_2)$ cones $\{
  H(0,\theta_i,1-c_2) \}_{i=1}^m$. Observe that $H(0,\theta_i,1-c_2)$ is a cone to the
  direction $\theta_i \in S^{d-1}$ with a small opening angle.

  For each point $y \in A$ denote the centres of the holes obtained
  from the $k$-porosity on the scale $r$ by $y_1,\dots,y_k$. Thus, $A
  \cap B(y_i,\roo r) = \emptyset$ and $|y_i-y| + \roo r \le r$ for
  every $i$, and $(y_i-y) \cdot (y_j-y) = 0$ whenever $i \ne j$. We
  observe that $A$ may be divided into $m^k$ sets of the form
  \[
    A_\iii = \bigl\{y \in A : y_j-y \in H(0,\theta_{i_j},1-c_2) \text{ for every } j
    \in \{ 1,\dots,k \} \bigr\}.
  \]
  where $\iii = (i_1,\ldots,i_k) \in \{1,\ldots,m\}^k$. Since $(y_i-y)
  \cdot (y_j-y) = 0$ for all $y \in A$ and all $i \neq j$, it follows
  that actually most of the sets $A_\iii$ are empty. Fix $\iii$ so
  that $A_\iii \neq \emptyset$ and choose $x$ so that $A_\iii \cap
  B(x,c_3 r) \neq \emptyset$. Define
  \[
    M_j = B(x,2 c_3r) \cap \partial\biggl( \bigcup_{y \in A_\iii \cap
      B(x,c_3 r)} B(y_{j},\roo r) \biggr)
  \]
  for all $j \in \{1,\ldots,k \}$ and let
  \[
    M = \bigcap_{j=1}^k M_j.
  \]
  Here $\partial C$ is the topological boundary of a given set $C$.

  By simple (but rather technical) geometric inspections, we observe
  that if $c_1$, $c_2$, and $c_3$ are chosen small enough (depending
  only on $d$), then the following assertions are true: If $f$ is the
  orthogonal projection from $M$ to the $k$-dimensional linear
  subspace $\bigcap_{j=1}^k\theta_{i_j}^\perp$, then
  \[
    |f(y)-f(z)|\le |y-z|\le 2|f(y)-f(z)|
  \]
  for all $y,z\in M$, so $f$ is bi-Lipschitz with constant
  $2$. Moreover, $\dist(y, M)\le 2\sqrt{d}(1-2\roo)r$ for all $y\in
  A_\iii\cap B(x,c_3 r)$. These estimates easily imply that $B(x,c_3
  r)\cap A_\iii$ may be covered by $c_4(1-2\roo)^{k-d}$ balls of
  radius $(1-2\roo)r$, where $c_4$ depends only on $d$ and the choice
  of $c_3$. On the other hand, the set $A_\iii\cap B(x,r)$ is clearly
  covered by $2^{2d} c_{3}^{-d}$ balls of radius $c_3 r$ and finally
  $A$ is covered by less than $m^k 2^{2d} c_{3}^{-d} c_4
  (1-2\roo)^{k-d}$ balls of radius $(1-2\roo)r$.
\end{proof}

Next we turn the previous lemma into a homogeneity estimate.

\begin{lemma}\label{lma:geometric}
If $0<\roo<\tfrac12$ and $\mu$ is a measure on $\R^d$ such that $\mu(A)>0$, where $A \subset \{x\in\R^d : \por_k(\mu,x) > \roo\}$, then for
each $\eps > 0$ there is a Borel set $A_\eps \subset A$ with $\mu(A_\eps) > 0$ such that
\[
\limsup_{r \downarrow 0} \hom^{5}_{1-2\roo,\eps,r}(\mu,x) < c(1-2\roo)^{k-d}
\]
for every $x \in A_\eps$, where $c > 0$ depends only on $d$.
\end{lemma}

\begin{proof}
Let $\eps > 0$ and take $r_0>0$ so that the set
\[
A_\eps = \{x \in A : \por_k(\mu, x, r, \eps/2) \ge \roo \textrm{ for all } 0 < r < r_0\}
\]
has positive $\mu$-measure. Now take a density point $x \in A_\eps$ and a radius $0 < r \le r_0/5$ for which
\begin{equation}\label{eq:densityradius}
\frac{\mu\bigl( A_\eps\cap B(x,5r) \bigr)}{\mu\bigl( B(x,5r) \bigr)} > 1- \eps.
\end{equation}
Let $\BB$ be a $((1-2\roo)r)$-packing of $B(x,r)$ so that $\mu(B) >
\eps\mu\bigl( B(x, 5r) \bigr)$ for all $B\in\BB$. Write $A_\BB$ for
the centres of the balls in $\BB$.
For each $B \in \BB$ choose $y \in A_\eps \cap B$. Because of \eqref{eq:densityradius},
such a point $y$ exists. A direct calculation using the $k$-porosity at $y$ on the
scale $r$ implies that
\[\por_k(A_\BB,x,r)\ge\roo-2(1-2\roo),\]
where $x$ is the centre of $B$.
Since this holds for all $x\in A_\BB$, Lemma \ref{lma:kporcover}
implies that $A_\BB$ may be covered by
$c\bigl(1-2\bigl(\roo-2(1-2\roo)\bigr)\bigr)^{k-d} = 5^{k-d}c(1-2\roo)^{k-d}$
balls of radius $5(1-2\roo)r$. Here $c = c(d)$ is the constant of Lemma
\ref{lma:kporcover}. It now follows that $\#\BB = \#A_\BB \leq 10^d5^{k-d}c(1-2\roo)^{k-d}$ yielding the claim.

It is important to note here that we are not
covering the set $A_\eps$ as it generally is
not even porous.
\end{proof}

\begin{proof}[Proof of Theorem \ref{thm:kpor}]
Let $1<c = c(d)<\infty$ be the constant of Lemma
\ref{lma:geometric} and let $0<\roo<\tfrac12$.
From the proof of Theorem \ref{thm:main} we observe that there
exists a constant $0<c_1 = c_1(d)<1$ so that for any $0 < m < s$
the choice $\delta_0 = c_{1}^{1/(m-s)}$ will suite as
$\delta_0=\delta_0=(m,s,5,N_d)$  in the claim of Theorem \ref{thm:main}.
Our aim is then to apply Theorem \ref{thm:main} with
\[
m = d - k + \frac{\log c}{-\log(1-2\roo)}, \quad s = m +
\frac{\log c_1}{\log(1-2\roo)},
\]
and $\delta = 1-2\roo$. Let $t = (m+s)/2$ and take
$M=M(N_d,\frac{1}{10})$ from Lemma \ref{thm:covering_thm}\eqref{covering4}. Here
$N_d$ is the doubling constant of $\R^d$.

Let $\delta_0 = \delta_0(m,s,N_d)$ be the constant in Theorem \ref{thm:main}.
Because we chose the parameters so that
\[
 \delta_0 \ge c_{1}^{\frac1{m-s}} = 1-2\rho = \delta,
\]
we may apply Theorem \ref{thm:main}.
Let $\eps$ to be the constant $\eps_0 = \eps_0(m,s,N_d,\delta)$ of Theorem \ref{thm:main}.

Proving the theorem now easily reduces to showing that
$\udimloc(\mu,x)\le s$ almost everywhere on the set
$A=\{y\in\R^d : \por_k(\mu,y)>\roo\}$. We may assume that
$\mu(A)>0$ since otherwise there is nothing to prove.
Suppose to the contrary that there exists a set $A' \subset A$
with positive measure such that $\udimloc(\mu,x)>s$ for all $x \in A'$.
Using Lemma \ref{lma:geometric}, we find a set $A_\eps \subset A'$
with $\mu(A_\eps)>0$ so that
\[
\limsup_{r \downarrow 0} \hom^{5}_{1-2\roo,\eps,r}(\mu,x) <
c(1-2\roo)^{-d+k} = (1-2\roo)^{-m}
\]
for all $x \in A_\eps$. Now Theorem \ref{thm:main} implies that
$\udimloc(\mu,x)\le s$ for $\mu$-almost all $x \in A_\eps$. This
contradiction finishes the proof.
\end{proof}

\begin{remark}
A measure $\mu$ is called $(\roo,p)$-mean $k$-porous at $x$ if for all
$\varepsilon>0$ and for all sufficiently large $n$, there are at least $pn$ values
$l\in\{1,\ldots,n\}$ with
$\por_k(\mu,x,2^{-l},\varepsilon)\ge\roo$. It follows from
\cite{RajalaSmirnov2007} that for any measure $\mu$ on $\R^d$, one has
$\udimloc(\mu,x)\le d-p-c(d)/\log(1-2\roo)$
for $\mu$-almost all $x\in\{y\in\R^d : \mu\text{ is
}(\roo,p)\text{-mean $1$-porous at }y\}$.
In light of Theorem \ref{thm:kpor} it is natural to ask whether this holds also for
mean $k$-porous measures: If $\mu$ is a measure on $\R^d$, $k \in \{ 1,\ldots,d \}$,
$0<\roo<1/2$, and $0<p<1$, is it true that \[\udimloc(\mu,x)\le d-pk-c/\log(1-2\roo)\]
for $\mu$-almost all $x\in\{y\in X : \mu\text{ is }(\roo,p)\text{-mean }k\text{-porous at }y\}$?
An affirmative answer to this question was recently obtained in \cite{SahlstenShmerkinSuomala2011}.
\end{remark}

\subsection{Porous measures on regular metric spaces}\label{metricporo}

If we consider $k$-porosity with $k=1$ there is no orthogonality
condition on the direction of holes. By replacing the Euclidean
distance $|x-y_1|$ by $d(x,y_1)$ in the definition, it makes
perfect sense to investigate $1$-porosity, which we simply call
porosity, in a general metric space $(X,d)$.

If $X$ is an $s$-regular metric space, then for any $A \subset X$ with
$\inf_{x \in A}\por_1(A,x) \ge \roo$, the packing dimension of $A$ is at most
$s-c\roo^s$; see \cite[Theorem 4.7]{JarvenpaaJarvenpaaKaenmakiRajalaRogovinSuomala2007}.
Recall that $X$ is $s$-regular if there exists a measure $\nu$ on $X$ and constants $a,b > 0$ so that
\begin{equation} \label{eq:reg_meas}
  ar^s \le \nu\bigl( B(x,r) \bigr) \le br^s
\end{equation}
for all $x \in X$ and $0<r\le \diam(X)$. Our result for measures in this direction is the following.

\begin{theorem}\label{thm:metricporo}
  If $X$ is an $s$-regular metric space and $\mu$ is a measure on $X$, then
  \begin{equation*}
    \udimloc(\mu,x) \le s-c\por_1(\mu,x)^s.
  \end{equation*}
  for $\mu$-almost all $x\in X$, where $c>0$ depends only on $s$ and the constants $a$ and $b$ of \eqref{eq:reg_meas}.
\end{theorem}

In the proof of Theorem \ref{thm:kpor}, we used a known estimate for
$k$-porous sets via a density point argument. In the proof of Theorem \ref{thm:metricporo} we will only be able to use Lemma \ref{lemma:weakdpp} as the density point property is not true in every $s$-regular metric space.
To prove Theorem \ref{thm:metricporo}, we
recall the following estimate from \cite[Corollary
4.6]{JarvenpaaJarvenpaaKaenmakiRajalaRogovinSuomala2007}.

\begin{lemma}\label{lma:porolemma}
If $X$ is an $s$-regular metric space with an $s$-regular measure $\nu$, then there exist constants $c_1,c_2,c_3 > 0$ depending only on $s$ and the constants $a$ and $b$ of \eqref{eq:reg_meas} that satisfy the following: If $x \in X$, $r_p>0$, $0 < r < c_3 \min\{ r_p,\diam(X) \}$, $A \subset B(x,r)$, and $\por_1(A,y,r') \ge \roo>0$
for all $y \in A$ and $0<r'<r_p$, then
\[
 \nu\bigl( A(r'') \bigr) \le c_1\nu\bigl( B(x,r) \bigr)\Big(\frac{r''}{r}\Big)^{c_2\roo^s}
\]
for all $0 < r'' < r$.
\end{lemma}

Now we are ready to prove Theorem \ref{thm:metricporo}.

\begin{proof}[Proof of Theorem \ref{thm:metricporo}]
  Let $\nu$ be an $s$-regular measure on $X$ with $\spt(\nu)=X$ and
  let the constants $c_1,c_2,c_3>0$ be as in Lemma
  \ref{lma:porolemma}. Let $0<\roo<\tfrac12$ and choose $\delta' > 0$ so small that $\log(c_1b/a)/\log(1/\delta) < (a c_2 \roo^s)/(4^s b)$ for all $0 < \delta \le \delta'$. We are going to apply Theorem \ref{thm:main} with
  \begin{equation*}
    m' = s - \frac{c_2a}{2b}(\roo/4)^s + \frac{\log(c_1b/a)}{-\log\delta'}, \quad
    s' = s - \frac{c_2a}{4b}(\roo/4)^s,
  \end{equation*}
  and $0 < \delta < \min\{ 1,\roo \diam(X)/2,\delta',\delta_0 \}$, where $\delta_0 = \delta_0(m',s',N,10) > 0$ is as in Theorem \ref{thm:main}. Let $\eps > 0$ be the constant $\eps_0 = \eps_0(m',s',N,\delta) > 0$ from Theorem \ref{thm:main}.

  It is clearly sufficient to prove that we have $\udimloc(\mu,x) \leq s-c\roo^s$ for almost all $x\in A_\eps$, where
  \begin{equation*}
    A_\eps = \{ x \in X : \por_1(\mu,x,r,\eps/2) \ge \roo \text{ for all } 0 < r < r_0 \}.
  \end{equation*}
 We note that $A_\eps$ is a Borel set. (A careful inspection of the
 definitions shows that it is in fact closed.) Let $x\in A_\eps$ be such that
 \[
  \lim_{r \to 0}\frac{\mu\bigl(B(x,r) \setminus A_\eps \bigr)}{\mu\bigl( B(x,5r) \bigr)} = 0.
 \]
 Recall that by Lemma \ref{lemma:weakdpp} this is true for $\mu$-almost every $x \in A_\eps$

 Take
  $0 < r < \min\{ 1,r_0/8\}$ so small that
  \begin{equation} \label{eq:density2}
    \frac{\mu\bigl(B(x,2r) \setminus A_\eps \bigr)}{\mu\bigl( B(x,10r) \bigr)} < \eps.
  \end{equation}
  Our goal is to show that for any $(\delta r)$-packing $\BB$ of
  \begin{equation*}
    A = \bigl\{ y \in B(x,r) : \mu\bigl( B(y,\delta r) \bigr) > \eps\mu\bigl( B(x,10r) \bigr) \bigr\}
  \end{equation*}
  the set $A_\BB = \{ y \in A : y \text{ is the centre point of some }
  B \in \BB \}$ satisfy the assumptions of Lemma
  \ref{lma:porolemma}. Using Lemma \ref{lma:porolemma}, we are
  able to estimate the cardinality of $\BB$ and hence also
  $\hom^{10}_{\delta,\eps,r}(\mu,x)$. The desired upper bound for
  $\udimloc(\mu,x)$ then follows from Theorem \ref{thm:main}.

  Fix a $(\delta r)$-packing $\BB$ of $A$ and $y \in A_\BB$. Assume
  first that $0 < r' < 2\delta r/\roo$. If $B(y,\roo'r/4) \setminus
  B\bigl( y,(\frac{a}{2b})^{1/s}\roo r'/4 \bigr) = \emptyset$,
  then it follows from the $s$-regularity of $\nu$ that
  \begin{equation*}
    a(\roo r'/4)^s \le \nu\bigl( B(y,\roo r'/4) \bigr) = \nu\bigl( B(y,(\tfrac{a}{2 b})^{1/s}\roo r'/4) \bigr) \le \tfrac{a}{2} (\roo r'/4)^s
  \end{equation*}
  which is impossible. Hence there exists a point $z \in
  B(y,\roo r'/4) \setminus B\bigl(
  y,(\tfrac{a}{2b})^{1/s}\roo r'/4 \bigr)$. Since $\roo
  r'/4<\delta r$, we have $A_\BB \cap B\bigl(
  z,(\tfrac{a}{2b})^{1/s}\roo r'/4 \bigr) = \emptyset$ and as
  $(\frac{a}{2b})^{1/s} \roo r'/4 + d(y,z) \le \roo r'/2 <
  r'$, it follows that $\por_1(A_\BB,y,r') \ge (\frac{a}{2b})^{1/s} \roo /4$ for all $0<r'<2\delta r/\roo $.

  Let us next assume that $2\delta r/\roo  \le r' \le 8r$. If $A_\eps \cap B(y,\delta r) = \emptyset$, then \eqref{eq:density2}
  and the definition of $A$ would imply that
  \begin{equation*}
    \mu\bigl( B(y,\delta r) \bigr) \le \mu\bigl( B(x,2r) \setminus A_\eps \bigr) < \eps\mu\bigl( B(x,10r) \bigr) \le \mu\bigl( B(y,\delta r) \bigr).
  \end{equation*}
  Hence there must be a point $z \in A_\eps \cap B(y,\delta r)$. The definition of $A_\eps$ in turn guarantees the existence of a point $w \in X$ such that $\mu\bigl( B(w,\roo r') \bigr) \le \frac{\eps}{2}\mu\bigl( B(z,r') \bigr)$ and $\roo r' + d(z,w) \le r'$. Now
  \begin{equation*}
    \roo r'/2 + d(y,w) \le \roo r'/2 + d(y,z) + d(z,w) \le \roo r' + d(z,w) \le r'
  \end{equation*}
  and $A_\BB \cap B(w,\roo r'/2) = \emptyset$ because for any $w' \in B(w,\roo r'/2)$ we
have $\mu\bigl( B(w',\delta r) \bigr) \le \mu\bigl( B(w,\roo r') \bigr) \le \frac{\eps}{2}\mu\bigl( B(z,r') \bigr) < \eps\mu\bigl( B(x,10r) \bigr)$, as $B(z,r')\subset B(x,10r)$.
Therefore $\por_1(A_\BB,y,r') \ge \roo/2$ for $2\delta r/\roo\le r'\le 8r$ and
consequently, for $2\delta r/\roo \le r' \le 4(\frac{b}{a})^{1/s}r$ we have
$\por_1(A_\BB,y,r') \ge  \min\{1, 2(\frac{a}{b})^{1/s}\}\roo/2$.

Now let $4(\frac{b}{a})^{1/s}r < r' < \diam(X)$ and put
$t=\frac14(\frac{a}{b})^{1/s}r'+2r$. Then
$t<\frac34(\frac{a}{b})^{1/s}r'$ and thus
\begin{equation*}
\nu\bigl( B(y,t) \bigr) \le b t^s<a \bigl(\tfrac{3r'}{4}\bigr)^s \le \nu\bigl( B(y,\tfrac{3}{4}r') \bigr).
\end{equation*}
So there exists $w\in B(y,\frac34 r')\setminus B(y,t)$. Now
$A_\BB \cap B\bigl( w,\frac14 (\frac{a}{b})^{1/s}r'\bigr) \subset
B(x,r) \cap B\bigl( w,\frac14 (\frac{a}{b})^{1/s}r' \bigr) = \emptyset$ and thus $\por_1(A_\BB,y,r') \ge \frac{1}{4}(\tfrac{a}{b})^{1/s}$.

  Putting the three estimates together, we have
  \begin{equation*}
    \por_1(A_\BB,y,r') \ge (\tfrac{a}{2b})^{1/s} \roo/4
  \end{equation*}
  for all $y \in A_\BB$ and $0<r'<\diam(X)$. We can now use Lemma \ref{lma:porolemma} to conclude
  \begin{equation*}
    \# \BB a(\delta r)^s \le \sum_{B \in \BB} \nu(B) = \nu\bigl( A_\BB(\delta r) \bigr) \le c_1\nu\bigl( B(x,r) \bigr) \delta^{\frac{c_2a}{2b}(\roo/4)^s} \le c_1b r^s \delta^{\frac{c_2a}{2b}(\roo/4)^s}
  \end{equation*}
  for all $0<r<c_3\diam(X)$. Since this is true for all $(\delta
  r)$-packings $\BB$ of $A$, and \eqref{eq:density2} is true for all
  small $r>0$, we get
  \begin{equation*}
    \limsup_{r \downarrow 0} \hom^{10}_{\delta,\eps,r}(\mu,x) \le \tfrac{c_1b}{a} \delta^{\frac{c_2a}{2b}(\roo/4)^s - s} < \delta^{-m'}
  \end{equation*}
  for $\mu$-almost every $x \in A_\eps$. Therefore, by Theorem
  \ref{thm:main}, we have $\udimloc(\mu,x) \le s'=s -
  \frac{c_2a}{4b}(\roo/4)^s$ for $\mu$-almost every $x \in
  A_\eps$. This completes the proof.
\end{proof}

\begin{remark}
Independently of our work, based on probabilistic ideas introduced in
\cite{HochmanShmerkin2012}, it was recently proved in \cite{Shmerkin2012} that
$\udimloc(\mu,x)\le d-c(d)p\roo^d$
for $\mu$-almost all $x\in\{y\in\R^d : \mu$ is $(\roo,p)$-mean $1$-porous at $y\}$
for measures in $\R^d$. It is natural to ask whether an analogous estimate
  \[\udimloc(\mu,x)\le s-cp\roo^s\] for $\mu$-almost all
  $x\in\{y\in X : \mu\text{ is
}(\roo,p)\text{-mean $1$-porous at }y\}$
is valid
in the setting of Theorem \ref{thm:metricporo} with a constant $c>0$ depending only on the $s$-regularity data.
This remains an open problem.
\end{remark}

\section{Examples and further remarks} \label{sec:examples}

So far in this article we have studied the relations between the local versions of $L^q$-spectrum, dimension and homogeneity,
and shown how these concepts can be used in estimating the dimension of measures.
Below, we give few straightforward
examples of situations where the local $L^q$-dimension seem to be
more reasonable than the global one. In \cite{KaenmakiRajalaSuomala2012b}, we show how the local $L^q$-spectrum can be used to develop local multifractal formalism.

In Examples \ref{ex:eka}--\ref{ex:q>h}, we use the fact that the $L^q$-spectrum
can be defined using the dyadic cubes. For the global spectrum this is well known
in the Euclidean setting and it is easy to see that this remains valid for the
local spectrum. For a detailed proof in the general metric setting we refer to \cite{KaenmakiRajalaSuomala2012b}.

\begin{example}\label{ex:eka}
We construct a probability measure $\mu$ on $\R^d$ so that for all
$0\le q<1$ we have
$\dim_q(\mu)=d$ while $\dim_q(\mu,x)=0=\dimloc(\mu,x)$ for
$\mu$-almost all $x\in\R^d$.

Our measure $\mu$ will be a countable sum of weighted Dirac
measures on $[0,1]^d$. Let us denote by $\QQ^n$ the dyadic subcubes of
$[0,1]^d$ of side-length $2^{-n}$.
At step 1, we let $n_1=1$ and attach a point mass of
size $2^{-d}$ to the centre point of all but one dyadic subcubes of
$[0,1)^d$ in
$Q\in\QQ^1$. Let $Q_1\in\QQ^1$ be the one remaining cube of measure
$2^{-d}$. At step 2 we choose a large integer $n_2\in\N$ and attach
a point mass of magnitude $2^{-n_2d}\mu(Q_1)$ to all but one of its
dyadic subcubes in $\QQ^{n_1+n_2}$. We continue inductively, at the
$k:$th stage we choose the one remaining cube
$Q_{k-1}\in\QQ^{n_1+\cdots+n_{k-1}}$, choose a large integer $n_{k}$ and
attach  a point mass of size $2^{-n_{k}d}\mu(Q_{k-1})$ to the centre
points of all but one dyadic subcubes of $Q_{k-1}$ in the collection
$\QQ^{n_1+\cdots+n_{k}}$.

At the $k$:th stage we have for all $0<q<1$ that
\begin{align*}
  \frac{1}{\log 2^{-n_k}} \log&\sum_{Q \in \QQ^{n_k}} \mu(Q)^q
  \le \frac{1}{\log 2^{-n_k}} \log \biggl(
  \sum_{\atop{Q \in \QQ^{n_k}}{Q\subset Q_{k-1}}}
  \mu(Q)^q \biggr) \\ &=\frac{\log\bigl(2^{n_k d(1-q)}\mu(Q_{k-1})^q\bigr)}{\log
    2^{-n_k}} = (q-1)d+q\,\frac{\log\mu(Q_{k-1})}{\log 2^{-n_k}}\,.
\end{align*}
Thus, choosing the numbers $n_k$ large enough,
we can ensure that
\[\tau_q(\mu)=\liminf_{n\to\infty}\frac1{\log 2^{-n_k}}\log\sum_{Q \in \QQ^{n_k}}
  \mu(Q)^q\le (q-1)d.\]
On the other hand, it is well known and easy to see that
$\tau_q(\mu)\ge(q-1)d$ for all measures $\mu$ on $\R^d$ with bounded
support; see Lemma \ref{tauprop}(2)). Therefore it follows that $\dim_q(\mu)=d$.
Furthermore, it is clear from the construction that
$\tau_q(\mu,x)=\dim_q(\mu,x)=\dimloc(\mu,x)=0$ for
$\mu$-almost all $x\in\R^d$.
\end{example}

\begin{example}\label{ex:toka}
If $\mu$ is the sum of a Dirac point mass at the origin and the
Lebesgue measure on the unit cube of $\R^d$, we see that
$\dim_q(\mu)=0$ whereas $\dim_q(\mu,x)=d=\dimloc(\mu,x)$ for all $q>1$ and all
$x\in[0,1]^d\setminus\{0\}$.
\end{example}

In $\R^d$, the $L^q$-spectrum estimates can be used directly to gain information on the
dimension of porous measures, but the results obtained this way are somewhat weaker
than the results obtained from the local homogeneity estimates in
\S \ref{sec:kpor} above; see Remark \ref{pororem}(3). One
motivation for investigating the local $L^q$-spectrum in metric spaces was to
find out, which of these two methods, if any, is stronger. Also,
in the view of Theorems \ref{thm:homodim} and \ref{thm:pointwise_dimq}, it
is interesting to compare $\dimhomo^\Lambda(\mu,x)$ to $\lim_{q \uparrow 1}
\dim_q(\mu,x)$. In the following two examples we show that there is no general relationship between these two values. We
present the examples in $\R$ but similar constructions work in any
dimension.
The first example also shows that a measure may have large homogeneity
even if it is of packing dimension zero.

\begin{example}\label{ex:h>q}
We construct an example in $\R$ so that $\lim_{q \uparrow 1}
\dim_q(\mu,x)=0$ while $\dimhomo(\mu,x)=1$ for $\mu$-almost all $x\in\R$.
The idea is to apply a construction resulting to a zero dimensional
measure on a Cantor set. The large homogeneity is obtained by
performing infinitely many (but extremely seldom so that it does not
affect the value of $\dim_q$) construction steps
where the measure is distributed almost uniformly inside the
construction intervals of that level.

We first pick a sequence
$0<\varepsilon_i\downarrow 0$ and then choose integers
$m_i,n_i\rightarrow\infty$ so that
\begin{align}
\frac{k+\sum_{j=1}^k
     m_j}{\sum_{j=1}^k
     (n_j+m_j)}<\varepsilon_k\label{nk}
\end{align}
for all $k\in\N$.
In the first step of the construction, we put
$\mu([0,2^{-n_1}]) = \mu([1-2^{-n_1},1]) = \tfrac12$. Then we
divide both intervals $[0,2^{-n_1}]$ and $[1-2^{-n_1},1]$ into
$2^{m_1}$ dyadic
subintervals of length $2^{-n_1-m_1}$ each getting $2^{-m_1}$ portion
of their parent's measure.

We continue the construction inductively. In the $k$:th step, we perform
the first step construction inside each of the construction
intervals of level $k$ just by replacing $n_1$ and $m_1$ with $n_k$ and $m_k$,
respectively.

As $m_k\rightarrow\infty$ it is clear that
$\hom_\delta^\Lambda(\mu,x)\approx\tfrac1\delta$ for all $x\in\spt(\mu)$ and
all small $\delta>0$. Thus $\dim_{\hom}^\Lambda(\mu,x)=1$ for all
$x\in\spt(\mu)$. On the other
hand, it follows easily from \eqref{nk}, that
$\tau_q(\mu,x)=\dim_q(\mu,x)=0=\dim_q(\mu)$ for
all $x\in\spt(\mu)$ and $0<q<1$.
\end{example}

\begin{example}\label{ex:q>h}
We construct an example in $\R$ so that $\lim_{q \uparrow 1}
\dim_q(\mu,x)=1$ but $\dim_{\hom}^\Lambda(\mu,x)=0$ for $\mu$-almost all
$x\in\R$. The idea is to perform a Cantor type construction resulting
to a zero dimensional measure, but add ``one-dimensional''
perturbation which affects only a dense set of measure zero, but
nevertheless, guarantees that the $\dim_q(\mu,x)$ is large for all
$x\in\spt(\mu)$.

Fix numbers $0<q_k\uparrow 1$ and
integers $n_k, l_k\in\N$ so that $n_k\rightarrow\infty$ and
$\sum_{k=1}^\infty 2^{-l_k}<\infty$.
In what follows, we choose a sequence of integers $m_k\rightarrow\infty$.
First of these, $m_1$, is taken so that
\[\frac{m_1(1-q_1)-l_1q_1}{n_1l_1+m_1}>\tfrac12(1-q_1).\]
The numbers $m_2,m_3,\ldots$ will be defined inductively below.

We begin the step $1$ of the construction by setting
$\mu([0,2^{-n_1}]) = \mu([1-2^{-n_1},1]) = \tfrac12$. Iterating this
in a self-similar manner for $l_1$ steps, we get
$2^{l_1}$ dyadic subintervals of $[0,1]$ of length $2^{-n_1 l_1}$ each
of measure $2^{-l_1}$.
We choose one of these intervals, say $I$, and
divide it into $2^{m_1}$ dyadic subintervals of length $2^{-m_1}|I|$ and
of measure $2^{-m_1}\mu(I)$. Inside the other $2^{l_1}-1$ construction
intervals of
length $2^{-n_1 l_1}$ we choose just the outermost subintervals of
length $2^{-l_1 n_1-m_1}$ and let both of these intervals have the
same measure (half of the measure of their parent).

In the beginning of the step $k$, $k\geq 2$, we have some amount, say
$I_1,\ldots,I_{N_k}$ dyadic intervals of equal length, denoted $2^{-M_k}$.
We perform the step $1$ construction inside each of these intervals, but
replace $n_1$, $l_1$, and $m_1$ by $n_k$, $l_k$, and $m_k$, respectively.
We choose $m_k$ so large that for each $I=I_j$, the
dyadic subintervals $J_i$ of $I$ of size $2^{-M_k-n_kl_k-m_k}$ chosen
in the construction satisfy
\begin{equation*}
\frac{\log\bigl(\sum_i\mu(J_i)^{q_k})\bigr)}{\log(2^{M_k+n_kl_l+m_k})} \ge \frac{\log\bigl(2^{m_k(1-q_k)}\bigl(2^{-l_k}\mu(I)\bigr)^{q_k}\bigr)}{\log(2^{M_k+n_kl_k+m_k})} > \frac{k}{k+1}(1-q_k).
\end{equation*}
The former estimate is obtained by summing over the range of intervals
where the measure was distributed uniformly.
As $q_k\uparrow 1$, we clearly get $\lim_{q\uparrow
   1}\dim_q(\mu,x)\geq 1$ for all $x\in\spt(\mu)$. On the other hand, as
$n_k\rightarrow\infty$, and $\sum_{k}2^{-l_k}<\infty$,
it follows that for $\mu$-almost all $x\in\R$, we have
$\hom_\delta^\Lambda(\mu,x)\leq C$ for all $0<\delta<1$ with some universal
constant $C>0$. Thus, in particular, $\dim_{\hom}^\Lambda(\mu,x)=0$ for almost all $x$.
\end{example}

\begin{remark}\label{rem:strict}
(1) From the previous example, it
follows that a strict inequality $\udimloc(\mu,x) < \lim_{q \uparrow
   1} \dim_q(\mu,x)$ is possible almost everywhere in Theorem
\ref{thm:pointwise_dimq}. We note that also
\begin{equation}\label{qdimeq}
    \lim_{q \downarrow 1} \dim_q(\mu,x) < \ldimloc(\mu,x)
  \end{equation}
is possible in a set of positive measure. A simple example is given by
letting $\mu=\mathcal{L}^1|_{[0,1]}+\sum_{n\in\N}2^{-n}\delta_{q_n}$
where $\mathcal{L}^1$ is the Lebesgue measure and $\{q_1,q_2,q_3,\ldots\}$
is dense in $[0,1]$. In order to get an
example where \eqref{qdimeq} holds almost everywhere, one can use a
similar idea as in Example \ref{ex:q>h} but this time one has to
construct a one dimensional measure with a dense zero dimensional
perturbation.

(2) We note that also the other inequalities in
  Theorem \ref{thm:pointwise_dimq} can be strict. For instance, see
  \cite[Proposition 3.1]{BatakisHeurteaux2002}.
\end{remark}

We finish the article by constructing a doubling metric space in which the density point property
does not hold. This space is then further modified in Examples \ref{ex:nodpp1} and \ref{ex:nodpp2}
to show that the inequalities in Theorem \ref{thm:entropy} may fail in a set of positive measure
without the density point property; see Remark \ref{rem:dppneeded}(2).

\begin{example} \label{ex:nodpp0}
Let $N_n$ be a sequence of integers and set $I_n = \{ 0,\ldots,N_n \}$.
We define an auxiliary function $f \colon \left(\N\cup\{0\}\right)^2\to [0,\infty)$ by
  setting
  \begin{equation}\label{eq:def_f}
    f(i,j) = f(j,i) =
    \begin{cases}
      0, &\text{if } i=j, \\
      2^{-i}, &\text{if } i \neq 0 \text{ and } j=0, \\
      (2^{-i}+2^{-j}), &\text{if } i,j \neq 0\text{ and }i \ne j.
    \end{cases}
  \end{equation}
We now set $\Sigma = \prod_{n=1}^\infty I_n$ and denote its elements by
$\iii = i_1i_2\cdots$, $\jjj = j_1j_2\cdots$, and so on. We
also denote $\Sigma_0=\{\varnothing\}$ and $\Sigma_n = \prod_{j=1}^n I_j$ for all $n \in \N$.
If $\iii\in \Sigma$ and $n \in \N$, then we let $\iii|_n = i_1\cdots i_n
\in \Sigma_n$. For $n \in \N$ and $\iii \in \Sigma_n$ we denote
$[\iii] = \{ \jjj \in \Sigma : \jjj|_n = \iii \}$. If $\iii,\jjj \in
\Sigma$ so that $\iii \ne \jjj$, then we let $\iii\wedge\jjj$ denote their longest common beginning. Let $|\iii|$ denote the length of a word $\iii$ (with the convention $|\varnothing|=0$) and $\iii\jjj$ the concatenation of two words $\iii,\jjj$ with $|\iii| < \infty$.

Let $\varepsilon_\varnothing=1$ and for $\iii\in\bigcup_n\Sigma_n$, let
\begin{equation}\label{metric}
\begin{cases}
0<\varepsilon_{\iii 0}\le 2^{-N_n}\varepsilon_\iii\,,\\
0<\varepsilon_{\iii i}\le 2^{-i}\varepsilon_\iii\text{ if }0\neq i\in I_{n+1}.
\end{cases}
\end{equation}
With these parameters we now define a distance $e \colon
  \Sigma \times \Sigma \to [0,\infty)$ on $\Sigma$ by setting
  \begin{equation*}
    e(\iii,\jjj) =
    \begin{cases}
      0, &\text{if } \iii,\jjj \in \Sigma \text{ so that } \iii = \jjj, \\
      \varepsilon_{\iii\wedge\jjj}f(i_{|\iii\wedge\jjj|+1},j_{|\iii\wedge \jjj|+1}), &\text{if }
      \iii,\jjj \in \Sigma \text{ so that } \iii \ne \jjj.
    \end{cases}
  \end{equation*}
This is indeed a distance: the triangle inequality follows easily from
\eqref{metric} and the definition of $f$.

Let us next show that $\Sigma$ is doubling. For this, we
choose $\iii\in\Sigma$, $0<r<\diam(\Sigma)\le 1$ and fix $n$ so that
$\eps_{\iii|_{n+1}}\leq r<\eps_{\iii|_n}$.
We also choose $k\in\N$ so
that $2^{-k}\eps_{\iii|_n}\leq r<2^{-k+1}\eps_{\iii|_n}$. If $k>1$, we get
$B(\iii,2r)\subset B(\iii,r)\cup B(\iii_0,r)\cup B(\iii_1,r)$, where
$\iii_0=i_1\cdots i_{n}0i_{n+2}\cdots$ and
$\iii_1=i_1\cdots i_{n}(k-1)i_{n+2}\cdots$. If $k=1$, then
$B(\iii,2r)\subset B(\iii,r)\cup B(\iii_2,r)\cup B(\iii_3,r)$, where
$\iii_2=i_1\cdots i_{n-1}0i_{n+1}\cdots$ and $\iii_3=i_1\cdots i_{n-1}N_{n}i_{n+1}\cdots$. In any case, we see
that $\Sigma$ is doubling with a doubling constant $3$.

To finish the construction, fix $N_n=n^3$ and let $\mu$ be a probability measure on $\Sigma$ that satisfies
\begin{equation}\label{mudef}
\begin{split}
      \mu([\iii0])  &= n^{-2}\mu([\iii]), \\
      \mu([\iii j]) &= N_{n}^{-1}(1-n^{-2})\mu([\iii])\,,
\end{split}
\end{equation}
for all $j \in \{1,\ldots, N_n\}$, $\iii\in\Sigma_{n}$, and $n \ge 2$.
If $A = \{ \iii\in\Sigma : i_j\ne 0 \text{ for all } j\in\N\}$, then
$\mu(A)>0$ since $\prod_{n=2}^\infty(1-n^{-2})>0$.

  Let $\iii\in A$ and define $r_{\iii,n} =
  \eps_{\iii|_{n}} 2^{-i_{n+1}}$ for all $n \in \N$. For each $\iii\in A$ it
  follows that $B(\iii,r_{\iii,n}) =
  [\iii|_{n+1}] \cup [\iii']$ for all $n \in \N$, where $\iii' =
  i_1\cdots i_{n}0 \in \Sigma_{n+1}$. Thus we get
\begin{equation}\label{mudpp}
\begin{split}
  \frac{\mu\bigl( A \cap B(\iii,r_{\iii,n}) \bigr)}{\mu(B(\iii,r_{\iii,n}))} &\le \frac{\mu([\iii|_{n+1}])}{\mu([\iii|_{n+1}])+\mu([\iii'])} \\
  &= \frac{N_{n}^{-1}(1-n^{-2}) \mu([\iii_{n}])}{\bigl(N_{n}^{-1}(1-n^{-2})+n^{-2}\bigr) \mu([\iii_{n}])} = \frac{1-n^{-2}}{1-n^{-2}+n}.
\end{split}
\end{equation}
In particular, as $n\rightarrow\infty$, we see that the density point property is not valid for $\mu$.
\end{example}

\begin{example}\label{ex:nodpp1}
In this example, we modify the previous example to obtain $\udim_1(\nu,\iii)>\udimloc(\nu,\iii)$ in a set of
positive measure. We continue with the same notation as in Example \ref{ex:nodpp0}. The space $\Sigma$ is
modified by gluing infinitely many small metric spaces into $A$:
Denote by $S$ the collection of all finite words $\iii\in\bigcup_{n=0}^\infty\Sigma_n$ that contain no zeros.
For each $\iii\in S$, let $(X_\iii,d_\iii)$ be a doubling metric space with diameter at most $\diam_e([\iii0])$
and with a uniform doubling constant (independent of $\iii$). Let $X=A\cup\bigcup_{\iii\in S}X_\iii$ and define
a distance $d$ on $X$ by
  \begin{equation} \label{eq:d_def}
    d(x,y) =d(y,x)=
    \begin{cases}
      e(x,y), &\text{if }x,y\in A,\\
      d_\iii(x,y), &\text{if }x,y\in X_{\iii},\\
      e(x,\iii000\cdots), &\text{if }x\in A\text{ and } y\in X_\iii,\\
      e(\iii000\cdots,\jjj000\cdots), &\text{if } x\in X_\iii,\, y\in X_\jjj\text{ and }\iii\neq\jjj.
     \end{cases}
  \end{equation}
Since $\diam_{d_\iii}(X_\iii)\le\diam_e([\iii0])$ and the doubling constant of $X_\iii$ is uniformly bounded,
it is readily checked that $(X,d)$ is a doubling metric space.

If $\mu$ is a measure on $\Sigma$ and $\nu_\iii$ are measures on $X_\iii$ with $\nu_\iii(X_\iii)=\mu([\iii0])$,
we define a measure $\nu$ on $X$ by setting
\begin{equation}
\nu=\mu|_A+\sum_{\iii\in S}\nu_\iii\,.
\end{equation}
Then $\nu|_A=\mu|_A$ and $\nu(X)=\mu(\Sigma)$.
Moreover, since $X_{\iii|_n}\subset B_{X}(\iii,r_{\iii,n}))$, $\nu(B_X(\iii,r_{\iii,n}))=\mu(B_\Sigma(\iii,r_{\iii,n}))$,
and $\nu(X_{\iii|_n})=\mu([\iii'])$, \eqref{mudpp} yields
\begin{equation}\label{eq:dppfails}
\frac{\nu(X_{\iii|_n})}{\nu(B(\iii,r_{\iii,n}))}\longrightarrow 1,
\end{equation}
as $n\rightarrow\infty$ for all $\iii\in A\subset X$.

We now specify $X_\iii$ and $\nu_\iii$: Let $X_\iii$ be a Euclidean interval of length $\diam_\Sigma([\iii0])$
and let $\nu_\iii$ be the length measure on $X_\iii$ normalized so that $\nu_\iii(X_\iii) = \mu([\iii 0])$. Then
\begin{equation*}
\lim_{\delta \downarrow 0} \int_{X_\iii} \frac{\log\nu\bigl( B(y,\delta) \bigr)}{\log\delta}\,d\nu(y)=\nu(X_\iii)
\end{equation*}
and combined with \eqref{eq:dppfails}, this yields
\begin{equation*}
\udim_1(\nu,\iii)\ge\lim_{n\rightarrow\infty}\limsup_{\delta \downarrow 0} \fint_{B(\iii,r_{\iii,n})} \frac{\log\nu\bigl( B(y,\delta) \bigr)}{\log\delta}\,d\nu(y) \ge 1.
\end{equation*}

All the above is valid for any choice of the $\varepsilon_\iii$, and by choosing them small enough, we
can easily guarantee that $\udimloc(\nu,\iii)=0$ for all $\iii\in A$. This proves that the latter estimate
of Theorem \ref{thm:entropy} may fail if the density point property is not satisfied.
\end{example}

\begin{example}\label{ex:nodpp2}
In this example, we modify the above examples to show that the density point property is needed also for the
first estimate of Theorem \ref{thm:entropy}. To obtain $\ldim_1(\nu,\iii)=0$ for $\iii\in A$ we simply can
replace the glued pieces $X_\iii$ in the previous example by singletons. But since we simultaneously want
$\ldimloc(\nu,\iii)>0$, we have to modify the construction such that on most scales, the measure $\nu$ is
very uniformly distributed.

Let $k_n$ be a strictly increasing sequence of integers and $J=\{k_n\,:\,n\in\N\}$. Let $N_{k_n}=n^3$ and
$N_n=2$ if $n\notin J$. For $n\in J$, let $I_n=\{0,\ldots, N_n\}$ and for $n\in\{0,1,2,\ldots\}\setminus J$,
let $I_n=\{1,2\}$ (so that $0\in I_n$ if and only if $n\in J$). Define $\Sigma$ as in Example \ref{ex:nodpp0}
and for $\iii\in\Sigma_n$, let
\begin{equation}\label{onmetriikkaa}
\begin{cases}
\varepsilon_{\iii 0}= 2^{-N_n}\varepsilon_\iii,\\
\varepsilon_{\iii i}= 2^{-i}\varepsilon_\iii\text{ for }0\neq i\in I_{n+1},
\end{cases}
\end{equation}
if $n\in J$ and
\begin{equation}\label{onmetriikkab}
\varepsilon_{\iii1}=\varepsilon_{\iii2}=\frac{\varepsilon_\iii}2
\end{equation}
otherwise.

Define a distance $e$ on $\Sigma$ by $e(\iii,\iii)=0$, and for $\iii\neq\jjj$, let
 \begin{equation*}
    e(\iii,\jjj) =
      \varepsilon_{\iii\wedge\jjj}f(i_{|\iii\wedge\jjj|+1},j_{|\iii\wedge \jjj|+1}),
  \end{equation*}
provided that $|\iii\wedge\jjj|\in J$ ($f$ is as in \eqref{eq:def_f}) and
\begin{equation*}
e(\iii,\jjj)=\frac{\varepsilon_{\iii\wedge\jjj}}{2}
\end{equation*}
otherwise. Again, it is a direct consequence of \eqref{eq:def_f} and \eqref{onmetriikkaa}--\eqref{onmetriikkab}
that $e$ is a distance.

Let $\mu$ be a probability measure on $\Sigma$ such that for $\iii\in\Sigma_{k_n}$, $n\ge 2,$
\begin{equation*}
\begin{split}
      \mu([\iii0])  &= n^{-2}\mu([\iii]), \\
      \mu([\iii j]) &= N_{n}^{-1}(1-n^{-2})\mu([\iii])\text{ for }0\neq j\in I_{k_n+1}\,,
\end{split}
\end{equation*}
and
\begin{equation*}
\mu([\iii1])=\mu([\iii2])=\frac{\mu([\iii])}{2}
\end{equation*}
if $|\iii|\notin J$.

As in the previous example, let $A$ (resp.\ $S$) be the collection of all infinite (resp.\ finite) words that contain no zeros.
For each $n\in\N$ and $\iii\in S\cap \Sigma_{k_n}$, let $X_{\iii}=\{x_\iii\}$ be a metric space consisting
solely of one point. Define $X=A\cup\bigcup_{n\in\N}\bigcup_{\iii\in S\cap \Sigma_{k_n}}X_\iii$ and
$\nu=\mu|_A+\sum_{n\in\N,\iii\in S\cap\Sigma_{k_n}}\mu([\iii0])\delta_{x_\iii}$. Let $d$ be a distance on $X$
defined via \eqref{eq:d_def}.

As in Example \ref{ex:nodpp1} above, it follows that $\nu(A)>0$, $X_{\iii|_n}\subset B(\iii,r_{\iii,n})$,
and that \eqref{eq:dppfails} holds for $\iii\in A$, $r_{\iii,n}=\varepsilon_{\iii|_{k_n}}2^{-i_{k_n+1}}$.
Moreover, a simple calculation implies
\begin{align*}
\lim_{\delta \downarrow 0} \int_{X_{\iii_{k_n}}} \frac{\log\nu\bigl( B(y,\delta) \bigr)}{\log\delta}\,d\nu(y)&=0,\\
\limsup_{\delta\downarrow0}\int_{B(\iii,r_{\iii,n})\setminus X_{\iii_{k_n}}}\frac{\log\nu\bigl( B(y,\delta) \bigr)}{\log\delta}\,d\nu(y)&\le C\nu(B(\iii_{k_n},r_{\iii,n})\setminus X_{\iii_{k_n}}),
\end{align*}
where $C>0$ depends only on the doubling constant of $X$. These estimates, together with \eqref{eq:dppfails}
imply that $\ldim_1(\nu,\iii)=0$ for all $\iii\in A$.

Again, the above holds regardless of the choice of $k_n$ and thus we can choose the sequence $(k_n)$ so that
\begin{equation}\label{eq:dimeq1}
\ldimloc(\nu,\iii)=1
\end{equation}
for all $\iii\in A$. To see this, observe first that if there were no sequence $(k_n)$, i.e.\ if $J=\emptyset$,
then it would be clear that
\begin{equation}
\ldimloc(\nu,\iii)=\ldimloc(\mu,\iii)=\liminf_{n\rightarrow\infty}\frac{\log\mu([\iii|_n])}{\log\varepsilon_{\iii|_n}}= 1.
\end{equation}
and since $N_{k_n}$, and the ratios $0<\diam([\iii i])/\diam([\iii])=\varepsilon_{\iii i}/\varepsilon_\iii$
for $\iii\in\Sigma_{k_n}$, $i\in I_{k_n}$ do not depend on the choice of $k_n$, we can choose $k_{n}\gg k_{n-1}$
inductively such that \eqref{eq:dimeq1} remains true.
\end{example}

\begin{ack}
  We are grateful to Marianna Cs\"ornyei for help in constructing Example \ref{ex:nodpp0}.
\end{ack}

\bibliographystyle{abbrv}
\bibliography{Bibliography}
\end{document}